\documentclass[twoside,12pt]{amsart}

\usepackage[foot]{amsaddr}

\author{Matthew D. Kvalheim}
\address[Kvalheim, Koditschek]
{School of Engineering and Applied Science, University of Pennsylvania,
	Philadelphia, PA 19104, USA}	
	
\author{Daniel E. Koditschek}

\email{kvalheim@seas.upenn.edu, kod@seas.upenn.edu}

\title[Necessary conditions for feedback stabilization and safety]{Necessary conditions for feedback stabilization and safety}

\usepackage{import} 

\usepackage[utf8]{inputenc}
\usepackage[T1]{fontenc}
\usepackage{lmodern}
\usepackage{amsmath}
\usepackage{amssymb}
\usepackage{amsthm}
\usepackage{stmaryrd} 
\usepackage{mathtools}
\usepackage{tikz-cd}
\usepackage{subfig}

\usepackage{thmtools}
\usepackage{thm-restate} 

\newcommand{\concept}[1]{\textbf{#1}}

\newcommand{\N}{\mathbb{N}}
\newcommand{\Z}{\mathbb{Z}}
\newcommand{\R}{\mathbb{R}}
\newcommand{\C}{\mathbb{C}}
\newcommand{\F}{\mathbb{F}}

\newcommand{\slot}{\,\cdot\,} 

\newcommand{\T}{T}
\newcommand{\D}{\T}
\newcommand{\id}{\textnormal{id}}
\newcommand{\rank}{\textnormal{rank}}
\newcommand{\interior}{\textnormal{int}}
\newcommand{\cl}{\textnormal{cl}}
\newcommand{\dom}{\textnormal{dom}}
\newcommand{\tor}{\mathbb{T}}
\newcommand{\sph}{\mathbb{S}}
\newcommand{\mob}{\mathbb{M}}
\newcommand{\klein}{\mathbb{K}}
\newcommand{\SE}{\mathsf{SE}}
\newcommand{\SO}{\mathsf{SO}}

\newcommand{\Hc}{\check{H}}
\newcommand{\Hom}{H}

\newcommand{\ctrl}{\mathcal{U}}
\newcommand{\st}{M}
\newcommand{\att}{A}
\newcommand{\cV}{\mathcal{V}}
\newcommand{\cW}{\mathcal{W}}
\newcommand{\cD}{\mathcal{D}}

\newcommand{\dist}[2]{\textnormal{dist}(#1, #2)}
\newcommand{\ip}[2]{\langle #1, #2 \rangle}
\DeclarePairedDelimiter\norm{\lVert}{\rVert}

\newtheorem{Lem}{Lemma}
\newtheorem{Th}{Theorem}

\newtheorem{Prop}{Proposition}

\newtheorem*{Th-non}{Theorem}
\newcommand{\Thnon}{Th-non}

\newcommand{\thistheoremname}{}
\newtheorem*{genericthm}{\thistheoremname}
{\renewcommand{\thistheoremname}{Theorem~\ref{#1}$'$}%
	\begin{genericthm}}
	{\end{genericthm}}

\theoremstyle{definition}
\newtheorem{Def}{Definition}
\newtheorem*{Def*}{Definition}
\newtheorem{Ex}{Example}

\newtheorem{Rem}{Remark}

\usepackage{marvosym}
\usepackage[
hypertexnames,%
citecolor=blue,%
colorlinks=true,%
linkcolor=red%
]{hyperref}
\usepackage[all]{hypcap}

\begin{document}
	
	\maketitle
	\begin{abstract}	
	Brockett's necessary condition yields a test to determine whether a system can be made to stabilize about some operating point via continuous, purely state-dependent feedback.
	For many real-world systems, however, one wants to stabilize sets which are more general than a single point.
	One also wants to control such systems to operate safely by making obstacles and other ``dangerous'' sets repelling.
	 
	We generalize Brockett's necessary condition to the case of stabilizing general compact subsets having a nonzero Euler characteristic in general ambient state spaces (smooth manifolds).
	Using this generalization, we also formulate a necessary condition for the existence of ``safe'' control laws.
	We illustrate the theory in concrete examples and for some general classes of systems including a broad class of nonholonomically constrained Lagrangian systems.
	We also show that, for the special case of stabilizing a point, the specialization of our general stabilizability test is stronger than Brockett's.
	\end{abstract}

    \tableofcontents	

\section{Introduction}\label{sec:intro}
In a seminal paper \cite{brockett1983asymptotic}, Brockett considered control systems of the form ($\dot{x}=dx/dt)$
\begin{equation}\label{eq:xdot-eq-fxu}
\dot{x} = f(x,u),
\end{equation}
and proved a beautiful theorem providing three necessary conditions for the existence of a continuously differentiable feedback law $u(x)$ rendering some specified point $x= x_0$ asymptotically stable. 
The third of these conditions, which we will simply refer to as \concept{Brockett's necessary condition}, is as follows.
Here $x\in \R^n$, $u\in \R^m$, and $f\colon \R^n\times \R^m\to \R^n$ is continuously differentiable.

\begin{\Thnon}[{\cite[Thm~1.(iii)]{brockett1983asymptotic}}]
If a continuously differentiable control law $u(x)$ rendering $x_0$ an asymptotically stable equilbrium exists, then the image of the mapping $(x,u)\mapsto f(x,u)$ contains a neighborhood of $0\in \R^n$. 
\end{\Thnon}
Brockett indicates one proof based on Wilson's converse Lyapunov function theorem \cite{wilson1967structure,wilson1969smooth,fathi2019smoothing} and the Poincar\'{e}-Hopf index theorem.
Since Wilson's theorem applies to merely continuous vector fields under the assumption of unique integrability, and since the Poincar\'{e}-Hopf theorem applies to merely continuous vector fields \cite{pugh1968generalized}\footnote{The statements of the Poincar\'{e}-Hopf theorem in \cite[pp.~32--35, p.~134]{milnor1965topology,guillemin1974differential} refer to smooth vector fields, but the proofs in these references can be used to prove the result for continuous vector fields with only superficial changes. Cf. \cite[p.~41]{milnor1965topology}: ``...our strong differentiability assumptions are not really necessary...''.}, the Poincar\'{e}-Hopf version of Brockett's proof actually yields the same result assuming only that the vector field $x \mapsto f(x,u(x))$ is continuous and has unique trajectories. 
That uniquely integrable continuous feedback suffices was also noted by \cite{zabczyk1989some}, who studied the problem via different techniques.

To paraphrase  \cite[pp.~1--2]{brockett1983asymptotic}, the theorem above is powerful enough to show that there is no continuous feedback law $(u,v)=(u(x,y,z),v(x,y,z))$ with unique trajectories making the origin asymptotically stable for the ``nonholonomic integrator'' or ``Heisenberg system'' \cite{bloch2015nonholonomic}
\begin{equation}\label{eq:heisenberg-intro}
\begin{split}
\dot{x} &= u\\
\dot{y}&= v\\
\dot{z} &= yu-xv. 
\end{split}
\end{equation}
This provides a counterexample to what was, in 1983, the oft-repeated conjecture that a reasonable form of local controllability implies the existence of a stabilizing control law \cite[p.~2]{brockett1983asymptotic}.

\subsection{Contributions and organization of the paper}
Motivated by problems arising in robotics and other settings of underactuated control systems, this paper introduces tests to determine if it is possible to use continuous feedback to make a system stabilize around---or, alternatively, operate safely relative to---some subset of state space.
One of our goals is to introduce practicable ``no-go'' theorems relieving the fruitless expenditure of computational resources \cite{papachristodoulou2005tutorial} searching for control Lyapunov functions \cite{sontag1999clf} or control barrier functions \cite{ames2017cbf,ames2019control} when they cannot exist.  
Another goal is to introduce mathematical tests precluding the emergence of hypothesized biological ``templates'' \cite{full1999templates,seipel2017conceptual} from hypothesized neuromechanical control architectures \cite{revzen2009towards}.

We generalize Brockett's necessary condition to one for the stabilization of any compact subset $\att$ of a smooth manifold $\st$ for a control system on $\st$, subject to the limitation that the Euler characteristic $\chi(\att)$ of $\att$ is nonzero.\footnote{In this paper, the Euler characteristic is defined using \v{C}ech-Alexander-Spanier cohomology; as we will discuss in \S \ref{sec:euler}, this Euler characteristic is always well-defined for a compact asymptotically stable subset of a continuous flow on a manifold, and it agrees with the standard Euler characteristic for submanifolds and other ``reasonable'' subsets.}
In the special case that $A$ is a point, we show (Ex.~\ref{ex:strictly-stronger}) that this necessary condition is stronger than Brockett's.
Our necessary condition for stabilization yields a corresponding condition for safety---an obstruction to the possibility of rendering a ``bad'' set $B$ repelling via continuous feedback, subject to the limitation again that the complement of $B$ is precompact with nonzero Euler characteristic. 

These results are powerful enough to prove the following additional facts about \eqref{eq:heisenberg-intro}.
(A portion of the first fact was previously noted in \cite{mansouri2007local}.) 
\begin{itemize}
\item If $\att$ is a compact connected submanifold with or without boundary, then there is no continuous feedback law $(u,v)=(u(x,y,z),v(x,y,z))$ with unique trajectories making $A$ asymptotically stable if $\att$ is not homeomorphic to a circle, cylinder, torus, M\"{o}bius band, or a $3$-dimensional submanifold with boundary (just because a compact $3$-dimensional submanifold \emph{without} boundary cannot embed in $\R^3$).
\item More generally, an arbitrary compact subset $\att$ cannot be rendered asymptotically stable by such a feedback law if the real \v{C}ech-Alexander-Spanier cohomology of $\att$ is not finite-dimensional (two examples are given below Def.~\ref{def:euler-cech} in \S \ref{sec:euler-asymp-stable}) or if the Euler characteristic of $\att$ defined using \v{C}ech-Alexander-Spanier cohomology is nonzero. 
\item If $S = \R^3\setminus B$ is bounded, where $B$ is some hypothetical ``bad'' set, then such a feedback law cannot ensure that $\cl(S)$ immediately flows into $\interior(S)$ if the real \v{C}ech-Alexander-Spanier cohomology of $S$ is not finite-dimensional or if the Euler characteristic of $S$ is nonzero. 
\end{itemize}
These facts also hold, e.g., for the ``kinematic unicycle'' model 
\begin{equation*}
\begin{split}
\dot{x} &= \cos(\theta)u\\
\dot{y}&= \sin(\theta)u\\
\dot{\theta} &= v 
\end{split}
\end{equation*}
on $\R^2 \times \sph^1$ commonly employed in the robot motion planning literature \cite{choset2005principles,bry2011rapidly,palmieri2014novel,park2015feedback,pacelli2018integration}  for modeling differential-drive or even legged  \cite{vasilopoulos2018sensor,vasilopoulos2020reactive,roberts2020examples} robots, and can be used to show that certain tasks of importance for applications are impossible to achieve using continuous feedback alone.
Results such as these underscore the importance of discontinuous (or ``hybrid'') \cite{bloch1992control, kolmanovksy1994discontinuous,khennouf1995construction, bloch1996stabilization,astolfi1996discontinuous,bloch2000stabilization,agrachev2001lie,prieur2005robust}   
or time-varying \cite{de1991exponential,coron1992global,pomet1992explicit,teel1992nonholonomic,m1993convergence,walsh1993stabilization,sordalen1995exponential,morin1999design,morin2000control,tian2002exponential,urakubo2015feedback,urakubo2018stability} stabilization strategies for control systems.

The remainder of this paper is organized as follows.
After discussing related work, we consider a version of the Euler characteristic of a compact asymptotically stable set in \S \ref{sec:euler} and explain why it is well-defined.
We establish our two main results (Theorems~\ref{th:generalized-brockett} and \ref{th:safety}) in \S \ref{sec:main} and illustrate them with examples in \S \ref{sec:examples}.
In \S \ref{sec:applications} we prove results showing that our necessary conditions for feedback stabilization and safety are not satisfied by some general classes of control systems including a class of nonholonomic Lagrangian systems.
In \S \ref{sec:compare} we compare Theorem~\ref{th:generalized-brockett} in the special case of point stabilization with results of Brockett and Coron. 
In \S \ref{sec:conclusion} we summarize our contributions and discuss prospects for future work.
The paper concludes with two appendices.
In App.~\ref{app:uniq-int-lyap} we review some facts about continuous and uniquely integrable vector fields, asymptotic stability, and Lyapunov functions.
In App.~\ref{app:low-dim-mfld-euler-zero} we determine which compact connected manifolds with boundary of dimension less than or equal to $2$ have vanishing Euler characteristic.  
\subsection{Related work}\label{sec:related-work}
Following the work of Brockett described above, \cite{zabczyk1989some} studied the problem of point stabilization using different techniques and observed that Brockett's necessary condition applies with uniquely integrable continuous state feedback, as opposed to the continuously differentiable feedback assumed in \cite[Thm~1.(iii)]{brockett1983asymptotic}.
Necessary conditions stronger than Brockett's for stabilization of equilibria were formulated in \cite{coron1990necessary} in terms of homology and in terms of stable homotopy groups; see also \cite[p.~292]{coron2007control}. 
In \cite{coron1992global} it was shown that, if time-varying feedback is allowed, then a certain accessibility condition implies the existence of smooth feedback stabilizing a point.
Further necessary conditions for  stabilization of an equilibrium based on \cite{coron1990necessary} were given in \cite{ishikawa1998equilibria}.
In \cite{orsi2003necessary} the observation of \cite{zabczyk1989some} was strengthened by further showing that unique integrability of the closed-loop vector field is not needed. 
We note that necessary conditions have also been obtained for asymptotic stabilization of a point in discrete-time \cite{lin1994design,kalabic2017mpc} and discontinuous \cite{ryan1994brockett,clarke1998asymptotic} systems, and for stabilization with an exponential convergence rate \cite{gupta2018linear,christopherson2020feedback}.

Most relevant to the present paper are necessary conditions for stabilizing sets more general than points via continuous feedback.
For control systems on $\R^n$, Byrnes stated a necessary condition reminiscent of Brockett's for rendering a compact set $\att\subset \R^n$ \emph{globally}\footnote{Note that being \emph{globally} asymptotically stable imposes strong topological restrictions; cf. Prop.~\ref{prop:cech-cohom-attractor}.
For example, any compact $2$-dimensional manifold (without boundary) smoothly embedded in $\R^3$ can be made asymptotically stable for the flow of some smooth vector field, but there does not exist a continuous flow on $\R^3$ rendering any of these submanifolds \emph{globally} asymptotically stable. 
Our generalization (Theorem~\ref{th:generalized-brockett}) of Brockett's necessary condition is for \emph{local} asymptotic stabilization, and is considerably broader than \cite[Cor.~4.1]{byrnes2008brockett} with respect to stabilization  while also affording a dual set of conclusions regarding safety which would be difficult even to translate into that ``global'' framework.} asymptotically stable \cite[Cor.~4.1]{byrnes2008brockett}.
In general, even for fully actuated systems, the mismatch between the topologies of a state space $M$ and subset $A\subset M$ precludes global stabilization of $A$ unless a ``cut'' is removed from $M$ \cite{baryshnikov2014centipede}, which is nonempty under mild assumptions when the stabilization problem's ``topological perplexity'' \cite{baryshnikov2021topological} is nonzero.
For the problem of \emph{local} asymptotic stabilization\footnote{\cite[Thm~12.5]{kappos1994role} states that a necessary condition for continuous feedback (not necessarily global) stabilization of a general compact subset $\att$ of $\R^{n+1}$ is that the (Riemannian \cite[Def.~12.3]{kappos1994role}) Gauss map $S\to \sph^{n}$ from a compact regular level set $S$ of a smooth Lyapunov function to the sphere $\sph^{n}$ is surjective.
However, the Gauss map of \emph{any} compact (boundaryless) hypersurface $S\subset \R^{n+1}$ is always surjective (since for any $v\in \sph^n$, the Riemannian Gauss map sends to $v$ any global maximizer of the ``height function'' $S\ni s\mapsto \ip{v}{s} \in \R$). 
Thus, while correct, this condition provides exactly zero stabilizability information. If $F$ is a closed-loop vector field rendering $\att$ asymptotically stable, one could instead consider the ``vector field Gauss map'' $\frac{F}{\norm{F}}\colon S\to \sph^n$ as in \cite{kappos1995necessary}, but examples of asymptotically stable periodic orbits in $\R^3$ show that surjectivity of this map is \emph{not} necessary for asymptotic stability.

Indeed, such an example is given by the asymptotically stable limit cycle with image $\att = \{r=1, z=0\}$ and basin $B(\att) = \{r>0\}$ for the smooth vector field $F = r(1-r^2) \partial_r + \partial_\theta - z \partial_z$ on $\R^3$.
Here $r^2 = x^2 + y^2, \quad \partial_r = (x\partial_x + y\partial_y)/r , \quad \text{and} \quad  \partial_\theta = x\partial_y - y \partial_x.$
Let $S\subset B(\att)$ be any regular level set of any smooth Lyapunov function for $\att$.
Since $F$ is nowhere-parallel to $\partial_z$ on $B(\att)$, the vector field Gauss map $\frac{F}{\norm{F}}\colon S\to \sph^2$ misses the north and south poles, so it is not surjective.
(At a high level, this example is possible because the Euler characteristic of $\att$ is $0$.)\label{foot:kappos}} 
we consider, Mansouri generalized Coron's homological necessary condition \cite{coron1990necessary} to one for stabilizing compact submanifolds of $\R^n$ \cite{mansouri2007local,mansouri2010topological}, and also introduced refined conditions for certain ``distributed'' stabilization problems \cite{mansouri2013topological,mansouri2015topological}.

For submanifolds of $\R^n$, we expect that the stabilizability tests furnished by \cite{mansouri2007local,mansouri2010topological} are stronger than those furnished by Theorem~\ref{th:generalized-brockett} (cf. \S\ref{sec:coron}).
However, Theorem~\ref{th:generalized-brockett} has the following advantages.
First, unlike the results of \cite{mansouri2007local,mansouri2010topological}, Theorem~\ref{th:generalized-brockett} applies to the stabilization of general compact subsets of general nonlinear smooth manifold state spaces.
In particular, the generality afforded by non-Euclidean state spaces enables the treatment of examples (such as Ex.~\ref{ex:kinematic-unicycle-1}, \ref{ex:satellite}, \ref{ex:vert-rolling-disk}) that do not satisfy the hypotheses of \cite{mansouri2007local,mansouri2010topological}.
Second, just as it is often simpler to apply \cite[Thm~1.(iii)]{brockett1983asymptotic} than the stronger \cite[Thm~2]{coron1990necessary}, it seems simpler to apply Theorem~\ref{th:generalized-brockett} than to apply the homological theorems of \cite{mansouri2007local,mansouri2010topological} in many cases.

Regarding our necessary condition for safety (Theorem~\ref{th:safety}) rather than stabilization, which is the second of our two main results, we are not aware of any closely related prior literature. 

\section{The Euler characteristic}\label{sec:euler}
An introductory discussion motivating the Euler characteristic is given in \S \ref{sec:euler-motivation}.
The reader familiar with (co)homology theory can proceed straight to \S \ref{sec:euler-asymp-stable}, where we define the Euler characteristic of an asymptotically stable set using \v{C}ech-Alexander-Spanier cohomology and review some relevant facts.   
\subsection{Motivation: the Euler characteristic of a CW complex}\label{sec:euler-motivation}
As motivation for our definition of the Euler characteristic, consider a topological space $X$ constructed as follows \cite[p.~519]{hatcher2001algebraic}.
\begin{itemize}
\item Start with a finite set $X^0 = \{x_1,\ldots,x_n\}$ of \concept{$0$-cells} equipped with the discrete topology.
\item Inductively, form the \concept{$n$-skeleton} $X^n$ from $X^{n-1}$ by attaching finitely many copies of the $n$-dimensional disk $D^n=\{x\in \R^n\colon \norm{x}\leq 1\}$ via continuous maps $\partial D^n \to X^{n-1}$ from the boundaries; the interiors of the copies of $D^n$ are called \concept{$n$-cells}.\footnote{E.g., $X^1$ is obtained from $X^0$ by gluing the endpoints of finitely many intervals to the points comprising $X^0$.}  
\item Stop at some finite step $n\geq 0$ and define $X\coloneqq X^n$.
\end{itemize}
The space $X$ constructed in this way is called a \concept{finite CW complex} (or, sometimes, ``finite cell complex'') \cite[p.~5]{hatcher2001algebraic} of \concept{dimension} $n$.
A finite CW complex is, in particular, a compact metrizable space;  finite CW complexes are a useful and broad class of topological spaces. 
For example, every compact smooth manifold with boundary has a \concept{finite CW decomposition}, i.e., is homeomorphic to a finite CW complex \cite[Thm~3.3]{morita2001geometry}. 

The \emph{Euler characteristic} $\chi(X)$ of a finite CW complex $X$ is defined to be \cite[p.~6]{hatcher2001algebraic}:
\begin{equation}\label{eq:euler-CW}
\chi(X)\coloneqq (\textnormal{number of even-dimensional cells}) - (\textnormal{number of odd-dimensional cells}).
\end{equation}
A finite CW decomposition of a topological space $X$ is generally not unique.
However, the number on the right side of \eqref{eq:euler-CW} does not depend on the numbers of cells in any specific choice of finite CW decomposition of a topological space $X$; this justifies the undecorated notation $\chi(X)$.
One way to prove this is to show that, for a finite CW complex $X$ \cite[Thm~2.44, Cor~3A.4,~Prop~3A.5]{hatcher2001algebraic},
\begin{equation}\label{eq:euler-homology}
\chi(X)=\sum_{i=0}^\infty (-1)^i\dim \Hom_i(X;\R) =\sum_{i=0}^{\dim X} (-1)^i\dim \Hom_i(X;\R).
\end{equation}
Here $\Hom_i(X;\R)$ is the $i$-th singular homology group with real coefficients, a finite-dimensional real vector space.
It follows that $\chi(X)$ does not depend on the specific finite CW decomposition of $X$, because \eqref{eq:euler-CW} and \eqref{eq:euler-homology} are equivalent, and because $\Hom_i(X;\R)$ depends only on the topology of $X$.
The latter statement is also true of the $i$-th singular real \emph{co}homology group $\Hom^i(X;\R)$, and in fact one obtains an equivalent definition of $\chi(X)$ by replacing each $\Hom_i(X;\R)$ in \eqref{eq:euler-homology} with $\Hom^i(X;\R)$.

Using either \eqref{eq:euler-homology} or its cohomological variant just mentioned, one obtains a definition of $\chi(X)$ for \emph{any} topological space $X$ not necessarily admitting a CW decomposition, as long as $\dim \Hom_i(X;\R)$ (or $\dim \Hom^i(X;\R)$) is finite for all $i$ and zero for $i$ sufficiently large.
However, many spaces $X$ (including even certain compact subsets of $\R^n$; see the two examples mentioned after Def.~\ref{def:euler-cech} in \S\ref{sec:euler-asymp-stable}) do not satisfy these requirements.
To define $\chi(X)$ for such spaces, one possibility is to replace the singular homology groups in \eqref{eq:euler-homology} with those from an alternative (co)homology theory (of which there are several) which might be better behaved.
As will be explained in \S \ref{sec:euler-asymp-stable}, the \concept{real \v{C}ech-Alexander-Spanier cohomology} groups $\Hc^i(X;\R)$ will be of particular use in our setting (this terminology follows \cite[p.~371]{massey1991basic}).
\begin{Rem}\label{rem:euler-equiv-defns}
When $X$ is a paracompact, Hausdorff, and locally contractible space (such as a manifold or CW complex), $\dim \Hc^i(X;\R) = \dim \Hom^i(X;\R)=\dim \Hom_i(X;\R)$ \cite[pp.~334, 340]{spanier1966algebraic}.
Thus, for such a space $X$, a definition of the Euler characteristic equivalent to \eqref{eq:euler-homology} can be obtained by replacing $\Hom_i(X;\R)$ with $\Hc^i(X;\R)$ in \eqref{eq:euler-homology}.
As previously remarked, such a definition is also equivalent to \eqref{eq:euler-CW} in the case that $X$ is a finite CW complex.
\end{Rem}

\subsection{The Euler characteristic of an asymptotically stable set}\label{sec:euler-asymp-stable}
With the discussion in \S\ref{sec:euler-motivation} as motivation, we first state the definition of Euler characteristic that we will use.
We denote by $\Hc^i(X;\R)$ the $i$-th real \v{C}ech-Alexander-Spanier cohomology group of a topological space $X$ and refer the reader to \cite[Ch.~6, Sec.~ 4 and 5]{spanier1966algebraic} or \cite[Ch.~8]{massey1978homology} for the definition and an extensive treatment; alternatively, the reader is referred to \cite[Sec.~3]{gobbino2001topological} and \cite[pp.~371--375]{massey1991basic} for brief introductions aligned with our needs.

\begin{Def}\label{def:euler-cech}
Let $X$ be a topological space. 
Assume that the dimensions $\dim \Hc^i(X;\R)$ of the real \v{C}ech-Alexander-Spanier cohomology groups are finite and vanish for $i$ sufficiently large.
Then the \concept{Euler characteristic } $\chi(X)$ of $X$ is defined by
\begin{equation}\label{eq:euler-cech-def}
\chi(X)\coloneqq \sum_{i=0}^\infty (-1)^i\dim \Hc^i(X;\R).
\end{equation}
\end{Def}
Next, suppose that the compact subset $\att\subset \st$ is asymptotically stable for some continuous local flow on the manifold $M$.
For reasons which will be made clear, we would like to consider the Euler characteristic $\chi(\att)$ defined according to Def.~\ref{def:euler-cech}.
However, as examples such as 
\begin{equation}\label{eq:k-ex-1}
K = \{0\}\cup \{1/n\colon n\in \N\}\subset \R
\end{equation}
or $K\subset \R^2$ the ``Hawaiian earring''( or ``shrinking wedge of circles'') \cite[Ex.~1.25]{hatcher2001algebraic} show, arbitrary compact subsets of manifolds need not have finite-dimensional real \v{C}ech-Alexander-Spanier cohomology (or singular (co)homology).\footnote{The set $K$ of \eqref{eq:k-ex-1} has infinite-dimensional real \v{C}ech-Alexander-Spanier cohomology since $\Hc^0(K;\R)$ is isomorphic to the vector space of locally constant $\R$-valued functions on $K$ \cite[Thm~3.2(i)]{gobbino2001topological}, which is infinite-dimensional in the present case. The set $K$ of \eqref{eq:k-ex-1} also has infinite-dimensional singular cohomology since $H^0(K;\R)$ is isomorphic to $\bigoplus_{i\in I}\R$, where $I$ is the set of path components of $K$ \cite[Thm~3.2(ii)]{gobbino2001topological}, which is infinite in the present case. For $K$ the Hawaiian earring, \cite[Prop.~2.4]{eda2000singular} implies that $\Hc^1(K;\R)$ is infinite-dimensional; \cite[Thm~3.1]{eda2000singular} and the universal coefficient theorem for cohomology \cite[pp.~195--197]{hatcher2001algebraic} imply that $\Hom^1(K;\R)$ is also infinite-dimensional.}
Hence for general compact $K$, Def.~\ref{def:euler-cech} cannot be used to define $\chi(K)$.
However, if $K = \att\subset \st$ is asymptotically stable for a continuous and uniquely integrable vector field $F$, the situation is better: the real \v{C}ech-Alexander-Spanier cohomology of $\att$ is always finite-dimensional.\footnote{A uniquely integrable vector field is one whose maximal integral curves are unique; see App.~\ref{app:uniq-int} for more details.}

To see this, denote by $B(\att)$ the basin of attraction of $\att$ (App.~\ref{app:lyap}), let $\Phi\colon \dom(\Phi)\subset \R \times \st \to \st$ be the unique maximal continuous local flow generated by $F$ (Lem.~\ref{lem:cont-loc-int-vf-gen-local-flow} in App.~\ref{app:uniq-int}), and recall that $\R\times B(\att)\subset \dom(\Phi)$ so that the restriction $\Phi|_{\R\times B(\att)}$ is a flow rather than merely a local flow.
We may then appeal to a known result (for flows) to deduce that the real \v{C}ech-Alexander-Spanier cohomology of $\att$ is always isomorphic to that of its basin $B(\att)$ \cite[p.~28]{shub1974dynamical}, \cite[Rem.~2.3(b)]{hastings1979higher}, \cite[Thm~6.3]{gobbino2001topological}.\footnote{However, as counterexamples involving the \emph{Warsaw circle} show, $\att$ might not be homotopy equivalent to its basin $B(\att)$ if $\att$ is sufficiently pathological \cite{hastings1978shape,hastings1979higher,robbin1988dynamical,gunther1993every}.}

Next, let $V\colon B(\att)\to [0,\infty)$ be a proper $C^\infty$ Lyapunov function for $\att$ (Lem.~\ref{lem:converse-lyap} in App.~\ref{app:lyap}) and fix $c > 0$.
Then the sublevel set  $\st_c\coloneqq V^{-1}([0,c])$ is a compact $C^\infty$ manifold with boundary $V^{-1}(c)$, and $B(\att)$ deformation retracts onto $\st_c$ by following trajectories of $\Phi$ \cite[Thm~3.2]{wilson1967structure}.
Thus, $B(\att)$ is homotopy equivalent to $\st_c$, so the real \v{C}ech-Alexander-Spanier cohomologies of $B(\att)$ and $\st_c$ are isomorphic \cite[p.~240]{spanier1966algebraic}.
Thus, for each $i$ we have isomorphisms
 \begin{equation*}
\Hc^i(\att;\R) \cong \Hc^i(B(\att);\R) \cong \Hc^i(M_c;\R) \cong \Hom^i(M_c;\R),
\end{equation*} 
where the last term is real singular cohomology of $\st_c$.
Since the smooth manifold with boundary $\st_c$ is compact, each $\Hom^i(M_c;\R)$ is finite-dimensional and vanishes when $i > \dim M_c$.
Thus, we have obtained the following result.

\begin{restatable}[]{Prop}{CechCohomAttractor}\label{prop:cech-cohom-attractor}
Fix any proper $C^\infty$ Lyapunov function $V\colon B(\att)\to [0,\infty)$ for $\att$ (see Lem.~\ref{lem:converse-lyap} in App.~\ref{app:lyap}), $c > 0$, and define $\st_c\coloneqq \{x\in B(\att)\colon V(x)\leq c\}$.
The real \v{C}ech-Alexander-Spanier cohomology $\Hc^*(\att;\R)$ of $\att$ is finite-dimensional and isomorphic to the \v{C}ech-Alexander-Spanier cohomologies of both $\Hc^*(\st_c;\R)$ and $\Hc^*(B(\att);\R)$, where $B(\att)$ is the basin of attraction of $\att$.
Thus, the Euler characteristic $\chi(\att)$ is well-defined according to Def.~\ref{def:euler-cech}.
\end{restatable}
\begin{Rem}
We remind the reader of the following fact discussed in \S \ref{sec:euler-motivation}.
When $\att$ has a finite CW decomposition (or, more generally, is homotopy equivalent to a finite CW complex), then $\chi(\att)$ is just the number of even-dimensional cells minus the number of odd-dimensional cells, as in \eqref{eq:euler-CW}.    
\end{Rem}

For the hypotheses of our main results, it is only relevant whether the Euler characteristic of some set is zero or nonzero. 
The following well-known result completely characterizes those compact smooth manifolds with boundary having zero Euler characteristic; one of the two implications will be invoked in the proof of Theorem~\ref{th:generalized-brockett}.
For the statement, a vector field $F$ on a smooth manifold with boundary $M$ \concept{points strictly inward} at $\partial M$ if, for every $x\in \partial M$, $F(x)$ points strictly inward in the usual sense \cite[p.~118]{lee2013smooth}; thus, every vector field on a smooth manifold $M$ \emph{without} boundary vacuously points strictly inward at $\partial M = \varnothing$. 
\begin{Lem}[Poincar\'{e} and Hopf]\label{lem:poincare-hopf}
A compact smooth manifold $M$ with (or without) boundary $\partial M$ has vanishing Euler characteristic $\chi(M)=0$ if and only if there exists a continuous, nowhere-vanishing vector field $F$ on $M$ which points strictly inward at $\partial M$.
\end{Lem}
That the existence of such a vector field $F$ implies $\chi(M)=0$ follows directly from the Poincar\'{e}-Hopf theorem \cite{pugh1968generalized}.
We will not need the opposite implication, but it can be established by ``canceling'' the isolated zeros of a generic smooth and inward pointing vector field $F$ against one another; this can be accomplished using the Hopf degree theorem \cite[p.~146]{guillemin1974differential}, the Poincar\'{e}-Hopf theorem, and the fact that the isolated zeros of $F$ are contained in a single coordinate chart \cite{michor1994n}.

    \section{Main results: stabilization and safety}\label{sec:main}  
   In this section we state and prove our main results, which apply to control systems of the usual form \eqref{eq:xdot-eq-fxu}, as well as a more general class of control systems which we now describe.
   
   Let $\st$ be a smooth manifold (the state space).
   Loosely following \cite{brockett1977control}, we say that a \concept{control system} is a $4$-tuple $(\ctrl,\st,p,f)$ such that the following diagram commutes (so $f$ is fiber-preserving):
   \begin{equation}\label{eq:control-system}
   \begin{tikzcd}[row sep=2.5em]
   \ctrl \arrow[dr,"p",swap] \arrow[rr,"f"] & & \T \st \arrow[dl,"\pi"] \\
   & \st 
   \end{tikzcd}.
   \end{equation}
   Here $\ctrl$ is a set, $p$ is a surjective map, and $\pi$ is the tangent bundle projection.
   A (state-feedback) \concept{control law} is a section $u\colon \st\to \ctrl$ of $p$, which means that $p\circ u = \id_\st$.\footnote{We remark that control systems are sometimes also called ``open systems'' \cite{lerman2018networks}.}
   Given a control law $u$, the commutativity of \eqref{eq:control-system} implies that $f\circ u$ is a vector field since $\pi \circ (f\circ u) = (\pi\circ f)\circ u = p\circ u = \id_M$.
   We refer to the vector field $f\circ u$ on $\st$ as defining the \concept{closed-loop system}.
   Note that \eqref{eq:control-system} specializes to the form \eqref{eq:xdot-eq-fxu} in  the case that $\ctrl = \st\times \R^m$ and $p(m,u)=m$. 
   The generality afforded by \eqref{eq:control-system} is quite useful for modeling systems in which the set of admissible controls depends on the state.
  
   There are many common assumptions on $\ctrl$ and $p$; it seems most authors assume that $\ctrl$ is a smooth manifold and that $p$ is a smooth (i) vector bundle \cite{brockett1983asymptotic}, or (ii) fiber bundle \cite{van1981symmetries, grizzle1985structure, bloch2015nonholonomic}, or (iii) submersion \cite{lerman2018networks}.
   For our purposes, we can obtain more general results by not imposing further conditions on $\ctrl$ or $p$ (including continuity properties) except for surjectivity of $p$ (so that sections exist), so we will not do so unless explicitly specified.
   
   \begin{Def}\label{def:stabilizable}
   Let $(\ctrl, \st, p, f)$ be a control system and $\att\subset \st$ be a compact subset.
   We say that $\att$ is \concept{stabilizable} if there exists a control law $u\colon M\to \ctrl$ such that (i) the closed-loop vector field $f\circ u$ is continuous and uniquely integrable and (ii) $\att$ is asymptotically stable for $f\circ u$.
   \end{Def}
   If $\ctrl$ is a smooth manifold and both $f$ and $u$ are locally Lipschitz,\footnote{\emph{Local} Lipschitz-ness is a well-defined, metric-independent notion of maps between smooth manifolds which depends only on the smooth structures of the manifolds \cite[Rem.~1]{kvalheim2021existence}.} then the assumptions of continuity and unique integrability in Def.~\ref{def:stabilizable} are automatic (see Rem.~\ref{rem:loc-lip-implies-uniq-int} in App.~\ref{app:uniq-int}).
   \begin{Rem}\label{rem:discontinuous-control-bundle}
   In Def.~\ref{def:stabilizable} it is not assumed that $p$ is continuous or even that $\ctrl$ is a topological space.
   This is because the proofs of Theorem~\ref{th:generalized-brockett} and related results only require continuity of the closed-loop vector field $f\circ u$, a somewhat weaker property. 
   \end{Rem}
   The following is the first  of our two main results; it provides a necessary condition for stabilizability of a compact subset.
   
    \begin{Th}\label{th:generalized-brockett}
    Let $(\ctrl,\st,p,f)$ be a control system and $\att\subset \st$ be a compact subset.
    Assume that $\att$ is stabilizable.
    \begin{itemize}
    \item Then the Euler characteristic $\chi(\att)$ of $\att$ is well-defined according to Def.~\ref{def:euler-cech}.
   \item Assume additionally that $\chi(\att) \neq 0$.
   Then for any neighborhood $\cW\subset \st$ of $A$, there exists a neighborhood $\cV\subset \T \st$ of the zero section $0_{\T \st}\subset \T \st$ such that, for any continuous vector field $X\colon \cW\to \cV \subset \T \st$ on $\cW$ taking values in $\cV$,
   \begin{equation}\label{eq:generalized-brockett-adversary}
   f(p^{-1}(\cW))\cap X(\cW)\neq \varnothing.
       \end{equation}
    \end{itemize}
    \end{Th}
    We like to think of the vector fields $X$ in Theorem~\ref{th:generalized-brockett} as \concept{adversaries} which need to be ``defeated'' via intersection of their images with the image of $f$ somewhere.
   See Fig.~\ref{fig:adversaries}.
   Because adversaries violating \eqref{eq:generalized-brockett-adversary} are highly nonunique, it is often quite easy to find them in examples (see Ex.~\ref{ex:heisenberg-1}, \ref{ex:kinematic-unicycle-1}, \ref{ex:satellite}, \ref{ex:vert-rolling-disk} and Prop.~\ref{prop:aff-gen-arg}, \ref{prop:general-second-order-cs-on-bundle-theorem} \ref{prop:nonhol-lagrangian}), thus ruling out stabilizability via Theorem~\ref{th:generalized-brockett}.
    
   \begin{figure}
   	\centering
   	\def\svgwidth{1.0\columnwidth}
   	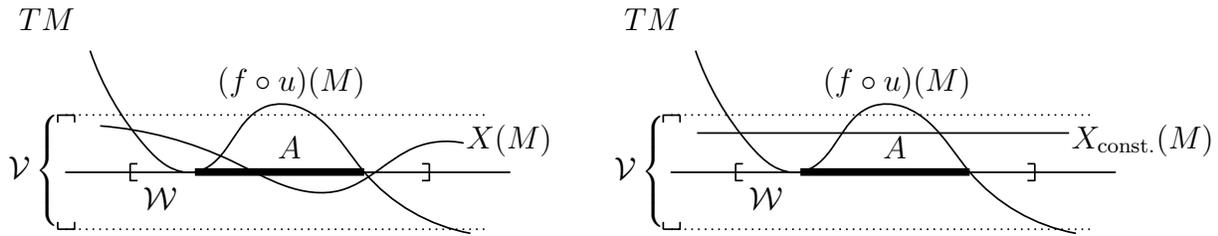
   	\caption{Depiction of the objects from Theorem~\ref{th:generalized-brockett} with a general continuous adversary $X$ and with a constant (with respect to some local trivialization of $\T M$) adversary (see Rem.~\ref{rem:generalize-brockett}).}\label{fig:adversaries}
   \end{figure}    

   \begin{figure}
   	\centering
   	\def\svgwidth{0.6\linewidth}
   	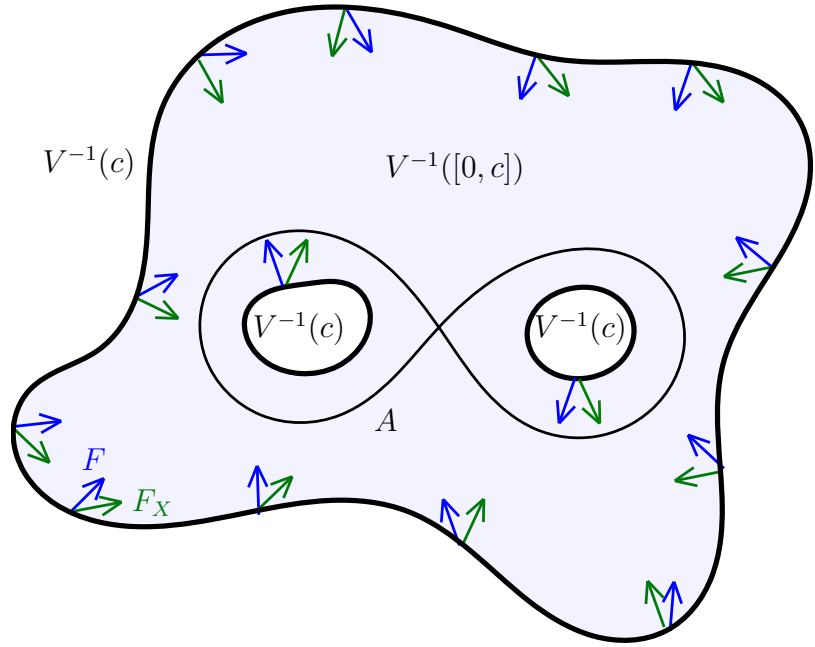
   	\caption{An illustration of the proof of Theorem~\ref{th:generalized-brockett}. Here $\att$ is homeomorphic to the wedge product of two circles, a ``figure eight'', which has Euler characteristic $-1\neq 0$.}\label{fig:thm1-proof}
   \end{figure}    
   
   \begin{Rem}\label{rem:templates}
   In biology, a ``template'' \cite{full1999templates,seipel2017conceptual} is often interpreted to mean an asymptotically stable invariant manifold $\att$ (carrying some prescribed restriction dynamics) of some closed-loop control system; hence
   Theorem~\ref{th:generalized-brockett} could enable one to rule out candidate templates from hypothesized neuromechanical control architectures \cite{revzen2009towards}.
   \end{Rem}
   
   \begin{Rem}[cf. Rem.~\ref{rem:cbf}]\label{rem:clf}
   For a broad subclass of control systems \eqref{eq:control-system}, existence of a ``control Lyapunov function'' \cite{sontag1999clf} for $\att$ implies existence of a continuous stabilizing feedback \cite{sontag1989universal}, so Theorem~\ref{th:generalized-brockett} can be used to rule out the existence of a control Lyapunov function.
   Ruling out the need to search for one might save valuable time and computational resources \cite{papachristodoulou2005tutorial}.
   \end{Rem}

    \begin{Rem}\label{rem:generalize-brockett}
    Theorem~\ref{th:generalized-brockett} is a generalization of Brockett's necessary condition \cite[Thm~1.(iii)]{brockett1983asymptotic}.
    To obtain Brockett's necessary condition from Theorem~\ref{th:generalized-brockett}, first specialize Theorem~\ref{th:generalized-brockett} by taking $\st = \R^n$, $\ctrl = \R^n \times \R^m$, and $p(x,u) = x$ (see \eqref{eq:control-system}).
    Then take $\att$ to be a singleton.
    Then weaken this specialization by restricting attention only to those adversaries $X$ (see Theorem~\ref{th:generalized-brockett}) which are \emph{constant} (using the canonical identification $\T\R^n \approx \R^n \times \R^n$ to view $X$ as $\R^n$-valued and thus define ``constant'').
    This yields \cite[Thm~1.(iii)]{brockett1983asymptotic}.
    If, in this special case, we did not restrict attention to constant adversaries $X$, then we show in Ex.~\ref{ex:strictly-stronger} that Theorem~\ref{th:generalized-brockett} is in fact strictly stronger than \cite[Thm~1.(iii)]{brockett1983asymptotic}.
    \end{Rem}

    \begin{Rem}\label{rem:brockett-proof}
    The following proof is inspired by one presented by Brockett \cite[Thm~1.(iii)]{brockett1983asymptotic} for stabilizing a point.
    It differs from Brockett's proof essentially only in the following two ways.
    First, $A$ is a singleton in Brockett's case, so it is immediate that $\chi(A) = 1\neq 0$ (see Eq.~\eqref{eq:euler-cech-def} and Rem.~\ref{rem:euler-equiv-defns}); however, in our case we need to refer to Prop.~\ref{prop:cech-cohom-attractor} to ensure that $\chi(A)$ is well-defined.
    Second, Brockett restricts attention to \emph{constant} adversaries $X\equiv c\in \R^n$ on $\R^n$, but we do not.
    (On a general smooth manifold $\st$, ``constant'' vector fields are not even well-defined).
    \end{Rem}

    \begin{proof}[Proof of Theorem~\ref{th:generalized-brockett}]
    Let $B(\att)$ be the basin of attraction of $\att$.
    Since $\att$ is asymptotically stable for the continuous and uniquely integrable vector field $F\coloneqq f\circ u$, it follows from Prop.~\ref{prop:cech-cohom-attractor} that $\chi(\att)$ is well-defined via Def~\ref{def:euler-cech}. 
    
    Next, assume that $\chi(\att)\neq 0$.
    Let $V\colon B(\att)\to [0,\infty)$ be a proper $C^\infty$ Lyapunov function for $\att$ (Lem.~\ref{lem:converse-lyap} in App.~\ref{app:lyap}) and fix $\epsilon>0$.
    Since the sublevel sets $\st_c\coloneqq V^{-1}([0,c])$ are compact, there exists $c > 0$ sufficiently small that $\st_c\subset \cW$.\footnote{Proof: the  family $(\st_1\setminus \st_c)_{c>0}$ of sets are open relative to $\st_1$ and form a cover of the compact set $\st_1\setminus \cW$.
    Thus, there is a finite subcover.
    Since $c_1<c_2$ implies $\st_{c_1}\subset \st_{c_2}$, it follows that there exists $0<c<1$ such that $\st_1\setminus \st_c \supset \st_1\setminus \cW$, which is equivalent to $\st_c \subset \st_1\cap \cW\subset \cW$.}
    By Prop.~\ref{prop:cech-cohom-attractor},
    \begin{equation}\label{eq:euler-equal-nonzero}
    \chi(M_c) = \chi(\att) \neq 0.
    \end{equation}
    
    Since the Lie derivative $L_F V(x) < 0$ when $x\not \in \att$, $L_F V(x) < 0$ for all $x\in \partial \st_c = V^{-1}(c)$.
    Since $\partial \st_c$ is compact, there exists $\epsilon > 0$ such that $L_F V(x) < -\epsilon$ for all $x\in \partial \st_c$.
    Continuity and compactness of $\partial \st_c$ imply the existence of a neighborhood $\cV$ of $0_{\T \st}$ such that, for any vector field $X$ on $\cW$ taking values in $\cV$, the translated vector field
    $$F_X\coloneqq F - X$$
    satisfies $L_{F_X}V < 0$.
    Thus, $F_X$ points strictly inward at $\partial M_c$ (Fig.~\ref{fig:thm1-proof}).
    By \eqref{eq:euler-equal-nonzero} and the Poincar\'{e}-Hopf theorem (Lem.~\ref{lem:poincare-hopf}), it follows that $F_X|_{M_c}$ has a zero for all such $X$.
    But a zero of $F_X|_{M_c}$ is a point  $x\in M_c\subset \cW$ such that $$0 = F_X(x) = F(x) - X(x) =  f\circ u(x) - X(x).$$
    Thus, for any $X$ taking values in $\cV$, there exists $x\in M_c\subset \cW$ such that $f\circ u(x) = X(x)$.
    This implies that $f(p^{-1}(\cW))\cap X(\cW)\neq \varnothing$ and completes the proof.
    \end{proof}

    Using Theorem~\ref{th:generalized-brockett}, we now proceed towards deriving a necessary condition for the existence of a ``safe'' control law rendering some ``bad'' set repelling (Theorem~\ref{th:safety}).

    \begin{Def}\label{def:strictly-inflowing}
    We say that a subset $S\subset \st$ is \concept{strictly positively invariant} for a continuous and uniquely integrable vector field $F$ on $\st$ if, for all $x_0\in \cl(S)$, the unique trajectory $t\mapsto x(t)$ of $F$ with initial condition $x(0)=x_0$ satisfies $x(t)\in \interior(S)$ for all $t>0$.
    \end{Def}
    Note that a sufficient condition for $S$ to be strictly positively invariant is that $S$ be a compact codimension-$0$ smooth submanifold with boundary such that $F$ points strictly inward at $\partial S$.

   \begin{Def}\label{def:safety}
   Let $(\ctrl, \st, p, f)$ be a control system and $S\subset \st$.
   We say that $S$ is \concept{savable} (or can be \concept{rendered safe}) if there exists a control law $u\colon M\to \ctrl$ such that (i) the closed-loop vector field $f\circ u$ is continuous and uniquely integrable and (ii) $S$ is strictly positively invariant for $f\circ u$.
   \end{Def}

    \begin{Lem}\label{lem:strict-inflowing-att-euler}
    Assume that $S\subset \st$ is precompact and strictly positively invariant for the continuous and uniquely integrable vector field $F$ on $\st$.
    Then $R\coloneqq \interior(S)$ contains a unique maximal (with respect to set inclusion) compact asymptotically stable subset $\att$.
    Moreover, the real \v{C}ech-Alexander-Spanier cohomology $\Hc^*(\att;\R)$ of $\att$ is finite-dimensional and isomorphic to both $\Hc^*(S;\R)$ and $\Hc^*(R;\R)$.
    Thus, the Euler characteristics $\chi(\att)$, $\chi(S)$, and $\chi(R)$ are well-defined according to Def.~\ref{def:euler-cech} and $\chi(\att)=\chi(S)=\chi(R)$. 
    \end{Lem}
    \begin{proof}
    Let $\Phi$ be the unique maximal continuous local flow of $F$ (see Lem.~\ref{lem:cont-loc-int-vf-gen-local-flow} in App.~\ref{app:uniq-int}).
    Since $\cl(S)$ is compact and strictly positively invariant, $[0,\infty)\times \cl(S)\subset \dom(\Phi)$, and there is a unique maximal compact asymptotically stable subset $\att\subset R$ given by
    $$\att\coloneqq \bigcap_{t>0}\Phi^t(R).$$
    See for example \cite[Sec.~II.5.C]{conley1978isolated} (cf. \eqref{eq:att-alt-def}).
    
    Let $\varphi\colon \st\to [0,\infty)$ be a $C^\infty$ function satisfying $\varphi^{-1}(0) = \st\setminus R$ \cite[Thm~2.29]{lee2013smooth} and define the uniquely integrable continuous vector field $\tilde{F}\coloneqq \varphi F$.
    Then $\att$ is asymptotically stable for $\tilde{F}$ with basin of attraction $R$.
    By Prop.~\ref{prop:cech-cohom-attractor} and Def.~\ref{def:euler-cech}, $\Hc^*(\att;\R)$ is finite-dimensional and isomorphic to $\Hc^*(R;\R)$; hence also the Euler characteristics of $\att$ and $R$ are well-defined via Def.~\ref{def:euler-cech} and are equal.
    
    To complete the proof, it suffices to show that $\Hc^*(S;\R)$ is isomorphic to $\Hc^*(R;\R)$.
    And to do this, by the homotopy invariance of \v{C}ech-Alexander-Spanier cohomology \cite[p.~240]{spanier1966algebraic}, it suffices to prove that $S$ is homotopy equivalent to its interior $R=\interior(S)$.
    Letting $\Phi$ be as above, define $g\colon S\to R$ via $g\coloneqq \Phi^1|_S$ and define $h\colon R\hookrightarrow S$ to be the inclusion map.
    Since $h\circ g = \Phi^1|_S$ and $g\circ h = \Phi^1|_R$, the maps $G\colon [0,1]\times S\to S$ and $H\colon [0,1]\times R\to S$ defined by $G\coloneqq \Phi|_{[0,1]\times S}$ and $H\coloneqq \Phi|_{[0,1]\times R}$ are continuous homotopies from $h\circ g$ to $\id_S$ and from $g\circ h$ to $\id_R$, respectively.
    Thus, $h$ is a homotopy equivalence with homotopy inverse $g$; this completes the proof.
    \end{proof}
    
   \begin{figure}
   	\centering
   	\def\svgwidth{1.0\linewidth}
   	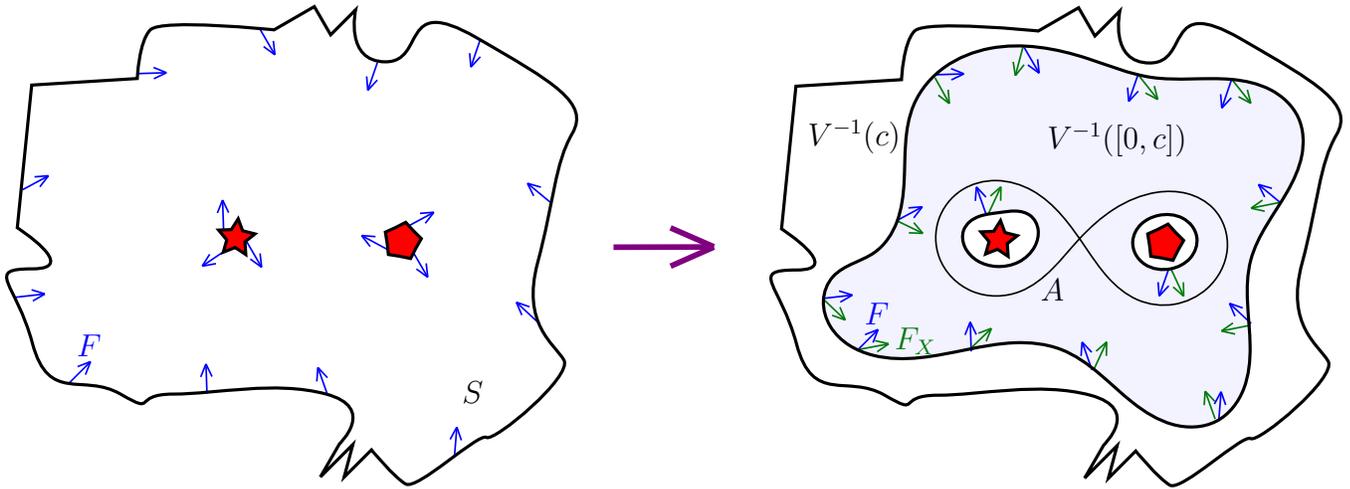
   	\caption{An illustration of the proof of Theorem~\ref{th:safety}. 
   	The idea is that, if $S$ is savable (Def.~\ref{def:safety}), then $S$ can be made strictly positively invariant (Def.~\ref{def:strictly-inflowing}) for some closed-loop vector field $F$, which implies (Lem.~\ref{lem:strict-inflowing-att-euler}) that there exists some asymptotically stable invariant set $\att \subset \interior(S)$ satisfying $\chi(\att) = \chi(S)$.
   	It follows that the situation in the proof of Theorem~\ref{th:generalized-brockett} becomes embedded within $\interior(S)$, as illustrated by Fig.~\ref{fig:thm1-proof} and the right side of the present figure.
   	This yields an immediate proof of Theorem~\ref{th:safety}.
   	In the present figure we are also emphasizing the fact that $\partial S$ does not need to be smooth.}\label{fig:thm2-proof}
   \end{figure}

   The following is the second of our two main results; it follows directly from Theorem~\ref{th:generalized-brockett} and Lem.~\ref{lem:strict-inflowing-att-euler} (see Fig.~\ref{fig:thm2-proof}).
   \begin{Th}\label{th:safety}
   Let $(\ctrl,\st,p,f)$ be a control system and $S\subset \st$ be a precompact subset.
   Assume that $S$ is savable.
   \begin{itemize}
   \item Then the Euler characteristic $\chi(S)$ is well-defined according to Def.~\ref{def:euler-cech}.
   \item Assume additionally that $\chi(S)\neq 0$.
   Then there exists a neighborhood $\cV\subset \T \st$ of $0_{\T \st}$ such that, if $X$ is any continuous vector field taking values in $\cV$, then $$f(p^{-1}(S))\cap X(S)\neq \varnothing.$$
   \end{itemize}
    \end{Th}
    
\begin{Rem}[cf. Rem.~\ref{rem:clf}]\label{rem:cbf}
   For a broad subclass of control systems \eqref{eq:control-system}, existence of a ``control barrier function'' \cite{ames2019control} for $S$ implies existence of a continuous safe feedback \cite{ames2017cbf}, so Theorem~\ref{th:safety} can be used to rule out the existence of a control barrier function.
   Ruling out the need to search for one might save valuable time and computational resources \cite{papachristodoulou2005tutorial}.
\end{Rem}

\section{Examples}\label{sec:examples}   
We illustrate Theorems \ref{th:generalized-brockett} and \ref{th:safety} on Ex.~\ref{ex:heisenberg-1} and \ref{ex:kinematic-unicycle-1} below, respectively, as well as in Ex.~\ref{ex:satellite}.
An important point demonstrated is that neither controllability nor accessibility \cite[pp.~201--202]{bloch2015nonholonomic} imply the existence of smooth (or continuous and uniquely integrable) feedback stabilizing or rendering safe various subsets of state space.
We also illustrate Theorems~\ref{th:generalized-brockett} and \ref{th:safety} on a nonholonomic mechanical system in Ex.~\ref{ex:vert-rolling-disk}, but we defer that example to \S \ref{sec:nonholonomic} in order to take advantage of Prop.~\ref{prop:nonhol-lagrangian}.

\begin{Ex}[Heisenberg system]\label{ex:heisenberg-1}
Consider the Heisenberg system
\begin{equation}\label{eq:heisenberg-ex}
\begin{split}
\dot{x} &= u\\
\dot{y}&= v\\
\dot{z} &= yu-xv 
\end{split}
\end{equation}
from \S \ref{sec:intro}, shown here again for convenience. 
As observed in \cite{brockett1983asymptotic}, the adversaries $X_\epsilon\coloneqq \epsilon \frac{\partial}{\partial z}$ are not in the image of the control system $f(x,y,z,u,v)$ defined by the right side of \eqref{eq:heisenberg-ex} for any $\epsilon \neq 0$.\footnote{Defining $f(x,y,z,u,v)$ to be the right side of \eqref{eq:heisenberg-ex}, denoting by $p\colon \R^n\times \R^m\to \R^n$ the projection onto the first factor, and identifying $f$ with a $\T \R^n \approx \R^n \times \R^n$ valued map, \eqref{eq:heisenberg-ex} defines a control system $(\R^n \times \R^m, \R^n, p, f)$ in the sense of \eqref{eq:control-system}.}
It follows from \cite[Thm~1.(iii)]{brockett1983asymptotic} that no point $x_*\in \R^3$ can be made asymptotically stable by a continuously differentiable control law $u(x,y,z), v(x,y,z)$.

Since $X_\epsilon\to 0$ uniformly on compact sets as $\epsilon \to 0$, the following stronger fact follows from Theorem~\ref{th:generalized-brockett}: no compact subset $\att\subset \R^3$ with a nonzero Euler characteristic $\chi(\att)$ (Def.~\ref{def:euler-cech}) can be made asymptotically stable by continuously differentiable feedback $u(x,y,z), v(x,y,z)$.\footnote{For example, if $\att\subset 
\R^3$ is a topological submanifold with (or without) boundary, then it is not possible to render $\att$ asymptotically stable by such feedback if $\att$ is not homeomorphic to one of the following: the circle $\sph^1$, the cylinder $\sph^1\times [0,1]$, the torus $\tor^2$, the M\"{o}bius band $\mob$, or a $3$-dimensional manifold with boundary; this is because these are the only topological manifolds with boundary embeddable in $\R^3$ which also have zero Euler characteristic (see App.~\ref{app:low-dim-mfld-euler-zero} and note that the Klein bottle does not topologically embed in $\R^3$ \cite[p.~256]{hatcher2001algebraic}).\label{foot:heisenberg-zero-euler-char-manifolds}}
Moreover, we do not need to assume that the control law is continuously differentiable or even locally Lipschitz; we just need to assume that the closed-loop vector field $(x,y,z)\mapsto f(x,y,z,u(x,y,z),v(x,y,z))$ is continuous and uniquely integrable.    

Of course, Theorem~\ref{th:safety} can also be applied to deduce interesting conclusions using the adversaries $X_\epsilon$, but to minimize redundancy we instead discuss Theorem~\ref{th:safety} in Ex.~\ref{ex:kinematic-unicycle-1} below.
\end{Ex}

\begin{figure}
	\centering
	\def\svgwidth{0.4\columnwidth}
\begingroup%
  \makeatletter%
  \providecommand\color[2][]{%
    \errmessage{(Inkscape) Color is used for the text in Inkscape, but the package 'color.sty' is not loaded}%
    \renewcommand\color[2][]{}%
  }%
  \providecommand\transparent[1]{%
    \errmessage{(Inkscape) Transparency is used (non-zero) for the text in Inkscape, but the package 'transparent.sty' is not loaded}%
    \renewcommand\transparent[1]{}%
  }%
  \providecommand\rotatebox[2]{#2}%
  \newcommand*\fsize{\dimexpr\f@size pt\relax}%
  \newcommand*\lineheight[1]{\fontsize{\fsize}{#1\fsize}\selectfont}%
  \ifx\svgwidth\undefined%
    \setlength{\unitlength}{383.4678973bp}%
    \ifx\svgscale\undefined%
      \relax%
    \else%
      \setlength{\unitlength}{\unitlength * \real{\svgscale}}%
    \fi%
  \else%
    \setlength{\unitlength}{\svgwidth}%
  \fi%
  \global\let\svgwidth\undefined%
  \global\let\svgscale\undefined%
  \makeatother%
  \begin{picture}(1,1)%
    \lineheight{1}%
    \setlength\tabcolsep{0pt}%
    \put(0,0){\includegraphics[width=\unitlength,page=1]{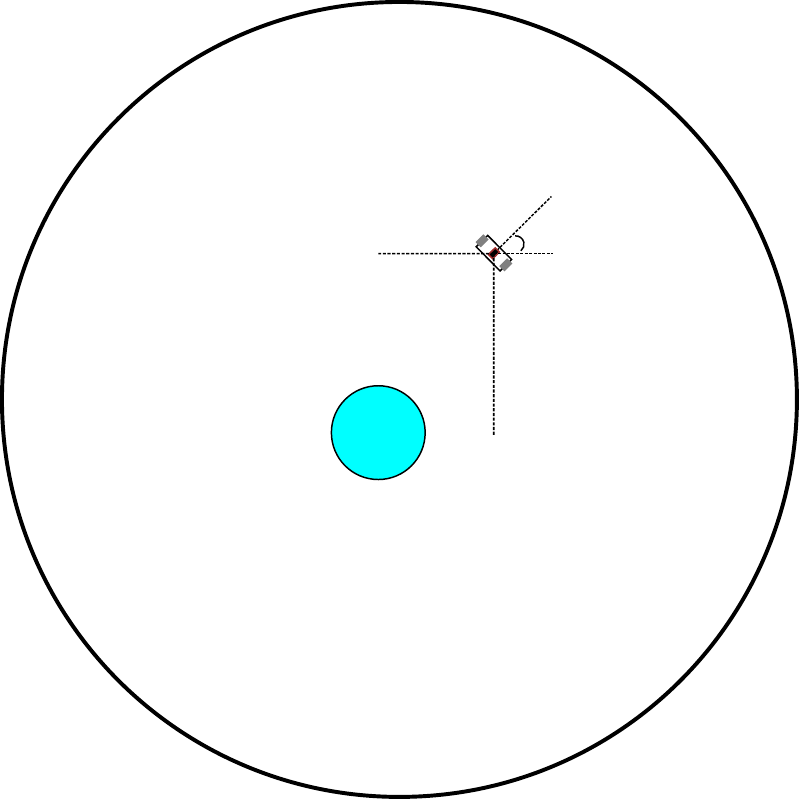}}%
    \put(0.68591159,0.6949397){\color[rgb]{0,0,0}\makebox(0,0)[lt]{\lineheight{1.25}\smash{\begin{tabular}[t]{l}$\theta$\end{tabular}}}}%
    \put(0,0){\includegraphics[width=\unitlength,page=2]{unicycle.pdf}}%
    \put(0.60767845,0.40156482){\color[rgb]{0,0,0}\makebox(0,0)[lt]{\lineheight{1.25}\smash{\begin{tabular}[t]{l}$x$\end{tabular}}}}%
    \put(0.42382988,0.67929278){\color[rgb]{0,0,0}\makebox(0,0)[lt]{\lineheight{1.25}\smash{\begin{tabular}[t]{l}$y$\end{tabular}}}}%
    \put(0,0){\includegraphics[width=\unitlength,page=3]{unicycle.pdf}}%
  \end{picture}%
\endgroup%

	\caption{This figure illustrates Ex.~\ref{ex:kinematic-unicycle-1}.
	It is desired for a differential-drive robot to simultaneously (i) aim the line of sight of a mounted camera to point within 90 degrees of the origin, (ii) avoid $n\geq 1$ obstacles, and (iii) remain inside a big disk and outside a small disk centered at the origin.
	For simplicity, in Ex.~\ref{ex:kinematic-unicycle-1} we assume that the size of the robot itself is negligible.
	(In the figure, the camera is aimed \emph{away} from the origin.)}\label{fig:unicycle}
\end{figure}

\begin{Ex}[Kinematic unicycle]\label{ex:kinematic-unicycle-1}
Consider the kinematic unicycle model
\begin{equation}\label{eq:kin-uni-ex}
\begin{split}
\dot{x} &= \cos(\theta)u\\
\dot{y}&= \sin(\theta)u\\
\dot{\theta} &= v 
\end{split}
\end{equation}
of a differential drive robot from \S \ref{sec:intro}, shown here again for convenience.
Let us imagine a differential drive robot mounted with a fixed-angle camera with which we want the robot to autonomously film an experiment happening at the origin $(x,y)=(0,0)\in \R^2$; see Fig.~\ref{fig:unicycle}.
The camera has a viewing angle of $180$ degrees, so we would like the robot not to face away from the origin: that is, we would like
\begin{equation*}
(x,y,\theta)\in S_0\coloneqq \{x\cos(\theta)+y\sin(\theta)\leq 0\}.
\end{equation*}
The camera should not get closer than some distance $\delta > 0$ to the origin, and the camera also has a finite range $R>0$, so we would also like 
\begin{equation*}
(x,y,\theta)\in S_1\coloneqq \{\delta^2 \leq x^2+y^2\leq R^2\}.
\end{equation*}
Next, we imagine that there are $n\geq 1$ obstacles $\mathcal{O}_1,\ldots, \mathcal{O}_n \subset \R^2$ contained in $\{(x,y)\colon \delta^2 < x^2 + y^2 < R^2\}$ which are bounded by continuous simple closed curves (i.e., images of continuous injective maps $\sph^1\to \R^2$ from the circle), and which we do not want the robot to come into contact with. 
For simplicity, we assume that the size of the robot itself is negligible. 
Finally, we would like to control \eqref{eq:kin-uni-ex} by purely state-dependent feedback, so as to generate robust, purely ``reactive'' behavior; we would also like the behavior to be deterministic and depend  continuously on the state (e.g., to avoid ``chattering''), so we would also like the control law to be time-independent and continuous, and such that the closed-loop vector field is uniquely integrable.

Define $\mathcal{O}\coloneqq \bigcup_{i=1}^n \mathcal{O}_i$, and define the set 
\begin{equation*}
S\coloneqq \{(x,y,\theta)\in S_0\cap S_1\colon (x,y)\not \in \mathcal{O}\}.
\end{equation*}
Then the desiderata of the preceding paragraph will be obtained if, in the terminology of Def.~\ref{def:safety}, $S$ is savable by a control law $u(x,y,\theta), v(x,y,\theta)$ (which, as part of Def.~\ref{def:safety}, requires that the closed-loop vector field be continuous and uniquely integrable).
Using Theorem~\ref{th:safety}, we show below that this seemingly reasonable task is impossible; some of the desiderata (e.g., continuity of the control law) must be relaxed.

To do this, we first observe that $S$ deformation retracts onto 
\begin{equation}\label{eq:S-def-retract}
\{(x,y,\theta)\in S\colon (\cos(\theta),\sin(\theta))=-(x,y)/\sqrt{x^2+y^2}\}
\end{equation}
via the continuous homotopy which simply turns the robot in place to face the origin (this homotopy is well-defined and continuous on $S$ since $S_0\supset S$).
Next, the set in \eqref{eq:S-def-retract} is homeomorphic to $\{(x,y)\in \R^2\setminus \mathcal{O}\colon \delta^2\leq x^2 + y^2 \leq R^2\}$, which is homotopy equivalent to a punctured disk with $n$ other disks removed from its interior.
The Euler characteristic of the latter set is equal to $1-1-n = -n$, and the Euler characteristic is a homotopy invariant, so
\begin{equation}\label{eq:kin-uni-euler}
\chi(S)=-n\neq 0.
\end{equation}

Next, notice that the adversaries $X_\epsilon\coloneqq \epsilon (\sin(\theta) \frac{\partial}{\partial x}-\cos(\theta)\frac{\partial}{\partial y})$ are not in the image of the control system $f(x,y,z,u,v)$ defined by the right side of \eqref{ex:kinematic-unicycle-1} for any $\epsilon \neq 0$.
Since $X_\epsilon\to 0$ uniformly on compact sets as $\epsilon \to 0$, \eqref{eq:kin-uni-euler} and Theorem~\ref{th:safety} imply that $S$ is not savable in the sense of Def.~\ref{def:safety}, as claimed above.
The same result follows for any set $S$, as long as $S$ is precompact and has a well-defined and nonzero Euler characteristic; the ``bad'' set $(\R^2\times \sph^1)\setminus S$ might represent physical and/or perceptual obstacles \cite{lopes2007visual} differing from those in the present example.

Theorem~\ref{th:generalized-brockett} can also be applied to deduce that no compact subset $\att \subset \R^2\times \sph^1$ with nonzero Euler characteristic $\chi(\att)$ is stabilizable (Def.~\ref{def:stabilizable}) for \eqref{eq:kin-uni-ex}.
See Footnote~\ref{foot:heisenberg-zero-euler-char-manifolds} for some consequences of this fact valid also for the present example. 
(However, for the present example, the fact that the Klein bottle $\klein$ does not topologically embed in $\R^2 \times \sph^1$ instead follows since $\R^2 \times \sph^1$ topologically embeds in $\sph^2\times \sph^1$ but $\klein$ does not \cite[Prop.~4.3]{jaco1970surfaces}).
\end{Ex}

The conclusions in the next example follow from Prop.~\ref{prop:aff-gen-arg} in \S \ref{sec:affine}, but we provide a self-contained analysis here.
\begin{Ex}[Satellite orientation]\label{ex:satellite}
Consider the Euler equations for the motion of a rigid body with two control torques but otherwise moving freely in space, imagined to describe a satellite actuated by two thruster jets \cite{byrnes1991attitude,bloch2015nonholonomic}: 
\begin{equation}\label{eq:euler}
\begin{split}
\dot{R} &= R\hat{\omega}\\
\dot{\omega} &= (I\omega)\times \omega + g_1(R,\omega)u_1 + g_2(R,\omega)u_2 .
\end{split}
\end{equation}
Here the state space $\st = \SO(3)\times \R^3$,  $R\in \SO(3)$ is a rotation matrix describing the orientation of a coordinate frame fixed in the body, $\omega\in \R^3$ is the vector of angular velocities with respect to the body frame, $I$ is the moment of inertia tensor with respect to the body frame, $\times$ is the cross product, the maps $(R,\omega)\mapsto g_i(R,\omega)\in \R^3$ are locally Lipschitz, and $\hat{(\slot)}\colon \R^3\to \mathfrak{so}(3)$ is the standard ``hat map'' defined by 
$$\hat{\omega}\coloneqq \begin{bmatrix}
0& -\omega_3 & \omega_2\\
\omega_3 & 0 & -\omega_1\\
-\omega_2 & \omega_1 & 0
\end{bmatrix},$$
where $\omega = (\omega_1,\omega_2,\omega_3)$.
For simplicity, we assume that $g_1(R,\omega)$ and $g_2(R,\omega)$ are linearly independent for all $(R,\omega)$.
For the case of thruster jets fixed to a satellite's body the $g_i$ would be constant, but treating the non-constant case is equally easy for us.

Let $(R,\omega)\mapsto X(R,\omega)$ be the vector field on $\st = \SO(3)\times \R^3$ given by $X = (0, g_1\times g_2)$, and consider the adversaries $X_\epsilon \coloneqq \epsilon X$ for $\epsilon \neq 0$.
We claim that $X_\epsilon$ is not in the image of the control system $$f(R,\omega,u_1,u_2) = (f_R(R,\omega,u_1,u_2), f_\omega(R,\omega,u_1,u_2))$$ defined by the right side of \eqref{eq:euler} for $\epsilon\neq 0$. 
To see this, first note that $f(R,\omega,u_1,u_2) = X_\epsilon(R,\omega,u_1,u_2)$ implies that $0 = f_R(R,\omega,u_1,u_2) = R\hat{\omega}$, which forces $\omega = 0$.
This in turn implies that $(g_1\times g_2)(R,\omega)$ must belong to the span of $g_1(R,\omega)$ and $g_2(R,\omega)$, but this is impossible since we have assumed that the $g_i$ are linearly independent everywhere.
This establishes the claim.

Since $X_\epsilon\to 0$ uniformly on compact sets as $\epsilon \to 0$, Theorem~\ref{th:generalized-brockett} implies that no compact subset $\att\subset \st$ with a nonzero Euler characteristic $\chi(\att)$ (Def.~\ref{def:euler-cech}) can be made asymptotically stable by locally Lipschitz feedback $u_1(R,\omega), u_2(R,\omega)$.\footnote{However, it is known \cite[Sec.~5]{byrnes1991attitude}---at least for certain choices of the $g_i$--- that it is possible to asymptotically stabilize topological circles, which have zero Euler characteristic.}
A special case of this conclusion is \cite[Cor.~1]{byrnes1991attitude}.
Similarly, it follows from Theorem~\ref{th:safety} that any precompact subset $S\subset \st$ having a well-defined (according to Def.~\ref{def:euler-cech}) and nonzero Euler characteristic is not savable (Def.~\ref{def:safety}).
\end{Ex}

\section{Applications}\label{sec:applications} 
In this section we discuss some implications of the results of \S \ref{sec:main}. 
An application to a class of affine control systems is given in \S \ref{sec:affine}, and an application to nonholonomic Lagrangian control systems is given in \S \ref{sec:nonholonomic}.
The latter application is illustrated on a model of a vertical rolling disk in Ex.~\ref{ex:vert-rolling-disk}. 

    \subsection{A class of systems affine in control}\label{sec:affine}
    Let $\st$ be a smooth manifold, and consider the control system
    \begin{equation}
    \label{eq:affine-control-system}
    \dot{x} = F(x) + \sum_{i=1}^Ng_i(x)u_i
    \end{equation}
    which is affine in the control inputs, where $F,g_1,\ldots, g_n$ are vector fields on $\st$ which are not assumed to be continuous.
    Formally, \eqref{eq:affine-control-system} defines a control system $(\ctrl,\st,p,f)$ in the sense of \eqref{eq:control-system} with $\ctrl = \st \times \R^N$, $f(x,u)$ equal to the right side of \eqref{eq:affine-control-system} with $u = (u_1,\ldots, u_N)$, and $p\colon M \times \R^N\to M$ projection onto the first factor.

    \begin{Prop}\label{prop:aff-gen-arg}
    Consider the control system defined by \eqref{eq:affine-control-system}.
    Let $\att\subset \st$ be a compact subset and $S\subset \st$ be a precompact subset.
    Assume there exists a continuous vector field $Y$ on $\st$ such that, for all $x\in \st$:
    \begin{align}
     F(x) \neq 0 & \implies F(x)\not \in \textnormal{span}\{Y(x),g_1(x),\ldots, g_N(x)\} \label{eq:aff-gen-arg-1}\\
    F(x) = 0 &\implies Y(x)\not \in \textnormal{span}\{g_1(x),\ldots, g_N(x)\}\label{eq:aff-gen-arg-2}.
    \end{align}     
    Then:
    \begin{itemize}
    \item if either (i) the Euler characteristic $\chi(\att)$ is not well-defined according to Def.~\ref{def:euler-cech}, or if (ii) $\chi(\att)$ is well-defined and $\chi(\att)\neq 0$, then $\att$ is not stabilizable (Def.~\ref{def:stabilizable}); and 
    \item  if either (i) the Euler characteristic $\chi(S)$ is not well-defined according to Def.~\ref{def:euler-cech}, or if (ii) $\chi(S)$ is well-defined and $\chi(S)\neq 0$, then $S$ is not savable (Def.~\ref{def:safety}).
    \end{itemize}
    \end{Prop}
    \begin{Rem}
     The dimensions of the spans in \eqref{eq:aff-gen-arg-1} and \eqref{eq:aff-gen-arg-2} are not assumed to be constant.
    \end{Rem}    
    \begin{Rem}\label{rem:byrnes-isidori}
    A very special case of Prop.~\ref{prop:aff-gen-arg} is the ``only if'' portion of \cite[Thm~1]{byrnes1991attitude}.
    \end{Rem}

    \begin{proof}[Proof of Prop.~\ref{prop:aff-gen-arg}]
    If the Euler characteristic of $\att$ is not well-defined, then the first conclusion of Theorem~\ref{th:generalized-brockett} implies that $\att$ is not stabilizable.
    Similarly, if the Euler characteristic of $S$ is not well-defined, then the first conclusion of Theorem~\ref{th:safety} implies that $S$ is not savable.
    Thus, we henceforth assume that $\chi(\att)$ and $\chi(S)$ are both well-defined and nonzero.
    
    For each $\epsilon > 0$ we define the continuous vector field $X_\epsilon \coloneqq \epsilon Y$ on $\st$.    
    We claim that 
    \begin{equation}\label{eq:aff-gen-arg-adversary-contra}
    \forall x\in \st, u\in \R^N, \epsilon \in (0,\epsilon_0)\colon X_\epsilon(x) \neq F(x) + \sum_{i=1}^N g_i(x)u_i.
    \end{equation}
    Indeed, \eqref{eq:aff-gen-arg-1} implies that $X_\epsilon(x) \neq F(x) + \sum_{i=1}^N g_i(x)u_i$ whenever $F(x)\neq 0$, and \eqref{eq:aff-gen-arg-2} implies the same when $F(x)=0$.
    This establishes \eqref{eq:aff-gen-arg-adversary-contra}.
    Since $X_\epsilon \to 0$ uniformly on compact sets as $\epsilon \to 0$, considering the adversaries $X = X_\epsilon$ in the final statements of Theorems \ref{th:generalized-brockett} and \ref{th:safety} implies that $\att$ is not stabilizable and $S$ is not savable.     
    \end{proof}

\subsection{Nonholonomically constrained Lagrangian systems}\label{sec:nonholonomic}
We first prove a general result about \emph{second-order} control systems on vector subbundles of a tangent bundle, of which the nonholonomic systems we consider are a special case.

Let $Q$ be a smooth manifold, $\pi_{\T Q}\colon \T Q\to Q$ be its tangent bundle projection and let $\cD\subset \T Q$ be a smooth vector subbundle or (constant rank) \concept{distribution}.
The \concept{rank} of $\cD$ is the dimension of its fibers.
Referring to \eqref{eq:control-system}, we consider here any control system of the form $(\ctrl, \st,p,f) = (\ctrl, \cD, p, f)$  satisfying
\begin{equation}\label{eq:second-order}
\D \pi_{\T Q}\circ f= \iota \circ p,
\end{equation}
where $f\colon \ctrl \to \T \cD\subset \T(\T Q)$, $p\colon \ctrl \to \cD$, $\D \pi_{\T Q}\colon \T(\T Q)\to \T Q$ is the tangent map (derivative) of $\pi_{\T Q}$, and $\iota\colon \cD\hookrightarrow \T Q$ is the inclusion.
We refer to such a control system as \concept{second order}.
To see why, let $(q^i, v^\mu)$ be the coordinates of a local trivialization induced by a smooth local frame $(e_\mu)_{\mu=1}^{\rank \cD}$ for $\cD$ defined over a chart for $Q$, and first note that the effect of the system evolving according to $f = \sum_{i}f^i \frac{\partial}{\partial q^i} + \sum_{\mu}f^\mu \frac{\partial}{\partial v^\mu}$ is the following: 
\begin{equation}\label{eq:vb-local-second-order-1}
\end{equation}
$$\dot{q}^i = f^i, \qquad \dot{v}^\mu = f^\mu.$$
Next, observe that \eqref{eq:second-order} and \eqref{eq:vb-local-second-order-1} imply $\sum_i \dot{q}^i \frac{\partial}{\partial q^i} = \sum_i f^i \frac{\partial}{\partial q^i}  = \sum_\mu v^\mu e_\mu$. 
Since we may express each $e_\mu$ as $e_\mu = \sum_{i}c^i_\mu \frac{\partial}{\partial q^i}$ with suitable smooth coefficients $c^i_\mu(q) \in \R$, it follows that $\dot{q}^i = \sum_{\mu}c^i_\mu v^\mu$; differentiating the latter expression with respect to time, it follows that \eqref{eq:vb-local-second-order-1} can be expressed as a constrained system of second order equations:
\begin{equation*}
\ddot{q}^i = \sum_{\mu} c^i_\mu f^\mu + \sum_\mu \sum_j \frac{\partial c^i_\mu}{\partial q^j}\dot{q}^jv^\mu, \qquad \dot{q}\in \cD,  
\end{equation*}
where $v_\mu$ can be expressed as a smooth function of the local coordinates $(q^i,\dot{q}^j)$ for $\T Q$, justifying the terminology.
   
\begin{Prop}\label{prop:general-second-order-cs-on-bundle-theorem}
Let $(\ctrl,\cD,p,f)$ be a second-order control system on a smooth vector subbundle $\cD\subset \T Q$ with $\rank \cD < \dim Q$.
Let $\att\subset \cD$ be a compact subset and $S\subset \cD$ be a precompact subset.
Assume there exists a continuous vector field $Y$ on $Q$ which is nowhere $\cD$-valued.
Then:
\begin{itemize}
\item if either (i) the Euler characteristic $\chi(\att)$ is not well-defined according to Def.~\ref{def:euler-cech}, or if (ii) $\chi(\att)$ is well-defined and $\chi(\att)\neq 0$, then $\att$ is not stabilizable (Def.~\ref{def:stabilizable}); and 
\item  if either (i) the Euler characteristic $\chi(S)$ is not well-defined according to Def.~\ref{def:euler-cech}, or if (ii) $\chi(S)$ is well-defined and $\chi(S)\neq 0$, then $S$ is not savable (Def.~\ref{def:safety}).
\end{itemize}

    \end{Prop}
    
    \begin{Rem}\label{rem:existence-of-transverse-vector-field}
    Note that a vector field $Y$ which is nowhere $\cD$-valued is, in particular, nowhere zero.
    Thus, in particular, if $Q$ is compact (and boundaryless, which we are assuming), then the existence of such a $Y$ implies that $\chi(Q)=0$.
    On the other hand, such a $Y$ exists if and only if the quotient vector bundle $\T Q/\cD$ admits a nowhere-zero section.
    Such a section exists if, for example, $Q$ is a  contractible space like $\R^n$ (since every vector bundle over a contractible base is trivial, i.e., isomorphic to a product bundle), but there are also many examples with $Q$ not contractible (e.g., let $\cD$ be a left or right invariant distribution on a Lie group $Q$).
    \end{Rem}
    
    \begin{proof}
    If the Euler characteristic of $\att$ is not well-defined, then the first conclusion of Theorem~\ref{th:generalized-brockett} implies that $\att$ is not stabilizable.
    Similarly, if the Euler characteristic of $S$ is not well-defined, then the first conclusion of Theorem~\ref{th:safety} implies that $S$ is not savable.
    Thus, we henceforth assume that $\chi(\att)$ and $\chi(S)$ are both well-defined and nonzero.

    By approximating $Y$ with a smooth vector field if necessary, we may assume that $Y$ is smooth \cite[Ch.~2.2]{hirsch1976differential}.
    Let the vector field $\tilde{Y}\colon \cD\to \T \cD$ be any lift of $Y$ to $\cD$ satisfying $\D \pi_{\T Q} \circ \tilde{Y} = Y\circ \pi_{\T Q}$ \cite[p.~202, Ex.~8-18]{lee2013smooth}.
    Defining $Y_\epsilon\coloneqq \epsilon Y$ and $\tilde{Y}_\epsilon\coloneqq \epsilon \tilde{Y}$ for $\epsilon > 0$,  $\tilde{Y}_\epsilon $ is the horizontal lift of $Y_\epsilon$ by linearity of horizontal lifts, so 
    \begin{equation}\label{eq:general-second-order-cs-on-bundle-theorem-1}
    \D \pi_{\T Q} \circ \tilde{Y}_\epsilon = Y_\epsilon\circ \pi_{\T Q}
    \end{equation}
    for all $\epsilon > 0$.
    On the other hand, since $(\ctrl,\cD,p,f)$ is second-order,
    \begin{equation}\label{eq:general-second-order-cs-on-bundle-theorem-2}
    \D \pi_{\T Q} \circ f = \iota \circ p.
    \end{equation}
    Since $Y_\epsilon$ is nowhere $\cD$-valued, the image of the right side of \eqref{eq:general-second-order-cs-on-bundle-theorem-1} is contained in $\T Q \setminus \cD$.
    On the other hand, the image of the right side of \eqref{eq:general-second-order-cs-on-bundle-theorem-2} is contained in $\cD$.
    Thus, examination of the left sides of \eqref{eq:general-second-order-cs-on-bundle-theorem-1} and \eqref{eq:general-second-order-cs-on-bundle-theorem-2} reveals that  $\tilde{Y}_\epsilon(\cD) \cap f(\ctrl) = \varnothing$. 
    Since $\tilde{Y}_\epsilon \to 0$ uniformly on compact subsets of $\cD$ as $\epsilon \to 0$, considering the adversaries $X = \tilde{Y}_\epsilon$ in the final statements of Theorems \ref{th:generalized-brockett} and \ref{th:safety} implies that $\att$ is not stabilizable and $S$ is not savable. 
    \end{proof}
    We now turn to nonholonomically constrained Lagrangian systems in a fairly general setting, although we do not strive for ultimate generality.
    Let $Q$ be a smooth manifold (the configuration space of a mechanical system, for example), let $L\colon \T Q\to \R$ be a smooth \concept{Lagrangian}, and let $\pi_{\T^* Q}\colon \T^* Q\to Q$ be the cotangent bundle projection.
    Following \cite[Def.~3.5.2]{abraham1987foundations}, the \concept{fiber derivative} $\F L\colon \T Q \to \T^* Q$ is a fiber-preserving map given in local coordinates by
    \begin{equation*}
    \F L\left(\sum_{i}\dot{q}^i \frac{\partial}{\partial q^i}\right)= \sum_{i}\frac{\partial L}{\partial \dot{q}^i}dq^i.
    \end{equation*}
    A Lagrangian is \concept{regular} if $\F L$ is a local diffeomorphism \cite[Def.~3.5.8]{abraham1987foundations}; in local coordinates, this means that the matrix $\frac{\partial^2 L}{\partial \dot{q}^i \partial \dot{q}^j}$ is everywhere invertible.

    Let $\cD\subset \T Q$ be a smooth vector subbundle representing the nonholonomic constraint $\dot{q}\in \cD$.
    Let $\ctrl$ be a smooth manifold, $p\colon \ctrl \to \cD$ be a surjective locally Lipschitz map, and let $F\colon \ctrl \to \T^*Q$ be a locally Lipschitz \concept{control force} satisfying $\pi_{\T^* Q} \circ F = \pi_{\T Q} \circ p$; this simply means that $F$ assigns to each $u\in \ctrl$ a covector in $\T^*_{q} Q$ over the same basepoint $q = \pi_{\T Q}(p(u))\in Q$ as that of $p(u)\in \cD_q$.
    Assuming that the constraint forces do no work, trajectories $q(t)$ of the system satisfy, in local coordinates, the \concept{Lagrange-d'Alembert equations} \cite[Eq.~1.3.6]{bloch2015nonholonomic} and constraint equations
    \begin{equation}\label{eq:nonhol-local}
    \frac{d}{dt}\frac{\partial L}{\partial \dot{q}^i}-\frac{\partial L}{\partial q^i} = \sum_{j}\lambda_j a^j_i + F_i \qquad \sum_j a^i_j(q) \dot{q}^j = 0,
    \end{equation}
    where $a^i = \sum_j  a^i_j(q) dq^j$ are locally defined $1$-forms satisfying $\mathcal{D} = \bigcap_i \ker a^i$ and the $\lambda_j(t)$ are \concept{Lagrange multipliers} determined to enforce satisfaction of the constraint $\dot{q}\in \cD$.
     
    Alternatively, we can write \eqref{eq:nonhol-local} in a global form.
    Let $\omega$ be the canonical symplectic form on $\T^* Q$, $\omega_L \coloneqq (\F L)^*\omega$ be the \concept{Lagrange $2$-form} \cite[Def.~3.5.5]{abraham1987foundations} on $\T Q$, $i_X \omega_L$ be the interior product of a vector field $X$ on $\T Q$ with $\omega_L$, and $E_L \coloneqq (\ip{\F L(v)}{v}-L) \colon \T Q\to \R$ the \concept{energy}.
    Then \eqref{eq:nonhol-local} is equivalent to (see
    \cite[Def.~3.5.11, Thm~3.5.17]{abraham1987foundations} or \cite[Eq.~3.4.6]{bloch2015nonholonomic}): 
    \begin{equation}\label{eq:nonhol-global}
    (dE_L - i_X \omega_L)_v \in (\D_v \pi_{\T Q})^*F + (\D_v \pi_{\T Q})^*\cD^0_{\pi_{\T Q}(v)}, \qquad v\in \cD,
    \end{equation}
    where\footnote{Here, note that $\dot{q}$ and $\ddot{q}$ are not local coordinate notations: $t\mapsto q(t)$ is a curve in $Q$, so $t\mapsto \frac{d}{dt}q(t)\eqqcolon \dot{q}(t)$ is a curve in $\T Q$, and thus $t\mapsto \frac{d}{dt}\dot{q}(t)\eqqcolon \ddot{q}(t)$ is a curve in $\T(\T Q)$.
    Thus, $\dot{q}\in \T Q$ contains both position and velocity information and $\ddot{q}\in \T(\T Q)$ contains position, velocity, and acceleration information.\label{foot:not-local-coord}} $\dot{q} = v\in \cD\subset \T Q$, $\ddot{q} = X\in \T \cD \subset \T(\T Q)$, and $\cD^0\subset \T^* Q$ is the annihilator of $\cD$.\footnote{A similar global formulation, but without an external force, can be found in \cite[Eq.~3.12]{monforte2004geometric}.
    However, the left side $(dE_L - i_X \omega_L)$ of \eqref{eq:nonhol-global} differs from that of \cite[Eq.~3.12]{monforte2004geometric} by a minus sign, which is necessary for the left side of \eqref{eq:nonhol-global} to coincide with that of \eqref{eq:nonhol-local} in local coordinates; cf. \cite[Eq.~3.4.6]{bloch2015nonholonomic}.}
    
    
    \begin{Prop}\label{prop:nonhol-lagrangian}
    Let $L\colon \T Q \to \R$ be a smooth Lagrangian which is regular, and let $\cD\subset \T Q$ be a smooth vector subbundle.
    \begin{enumerate}
    \item\label{item:nonhol-lagrangian-1}    Then \eqref{eq:nonhol-global} (or, equivalently, \eqref{eq:nonhol-local}) defines a second order control system $(\ctrl, \cD, p, f)$ such that a curve $q(t)$ satisfies \eqref{eq:nonhol-global} if and only if $\ddot{q}(t) = X =  f$,  and $f\colon \ctrl \to \T \cD$ is locally Lipschitz. 
    \item\label{item:nonhol-lagrangian-2} 
    Let $\att \subset \cD$ be a compact subset and $S\subset \cD$ be a precompact subset.
    Assume there exists a continuous vector field $Y$ on $Q$ which is nowhere $\cD$-valued. 
    Then:  
\begin{itemize}
\item if either (i) the Euler characteristic $\chi(\att)$ is not well-defined according to Def.~\ref{def:euler-cech}, or if (ii) $\chi(\att)$ is well-defined and $\chi(\att)\neq 0$, then $\att$ cannot be rendered asymptotically stable by any locally Lipschitz control force $F\colon \ctrl \to \T^* Q$ in \eqref{eq:nonhol-global}; and 
\item  if either (i) the Euler characteristic $\chi(S)$ is not well-defined according to Def.~\ref{def:euler-cech}, or if (ii) $\chi(S)$ is well-defined and $\chi(S)\neq 0$, then $S$ cannot be rendered strictly positively invariant (Def.~\ref{def:strictly-inflowing}; cf. Def.~\ref{def:safety}) by any locally Lipschitz control force $F\colon \ctrl \to \T^* Q$ in \eqref{eq:nonhol-global}.
\end{itemize}
    \end{enumerate}
    \end{Prop}
    \begin{Rem}
    With only minor changes, a more general result can be formulated assuming only that the force law produces a closed-loop vector field $X$ which is continuous and uniquely integrable. 
    However, the present formulation of Prop.~\ref{prop:nonhol-lagrangian} has the virtue of giving sufficient conditions under which $f\colon \ctrl \to \T\cD$ is locally Lipschitz, so that closed-loop vector fields $X$ produced by locally Lipschitz forces $F$ are locally Lipschitz and thus uniquely integrable. 
    \end{Rem}
    \begin{Rem}\label{rem:global-pfaffian-constraint}
    The nonholonomic distribution $\cD$ is \emph{locally} defined by the \concept{Pfaffian constraints} $\sum_j a^i_j(q) \dot{q}^j = 0$ of \eqref{eq:nonhol-local}, where the $a^i$ are locally-defined $1$-forms.
    Suppose that at least one of these $1$-forms can be \emph{globally} defined, i.e., suppose there exists a continuous $1$-form $\alpha$ on $Q$ with $\cD \subset \ker \alpha$.
    Then if the vector field $Y$ on $Q$ is the dual of $\alpha$ with respect to any Riemannian metric, $Y$ is nowhere $\cD$-valued, so the hypothesis in item \ref{item:nonhol-lagrangian-2} of Prop.~\ref{prop:nonhol-lagrangian} is satisfied.
    Since nonholonomic distributions are often described by Pfaffian constraints in practice, this provides a (often easy) means for verifying this hypothesis in examples. 
    Conversely, if there exists a continuous vector field $Y$ defined on all of $Q$ which is nowhere $\cD$-valued, then the metric dual of the orthogonal projection of $Y$ onto the orthogonal complement of $\cD$ (with respect to any Riemannian metric) is a continuous $1$-form $\alpha$ satisfying $\cD \subset \ker \alpha$.
    Thus, the existence of a continuous $Y$ which is nowhere $\cD$-valued is equivalent to the existence of a continuous $\alpha$ satisfying $\cD \subset \ker \alpha$; see Rem.~\ref{rem:existence-of-transverse-vector-field} for more discussion.
    \end{Rem}
    
    \begin{Rem}\label{rem:eq-submanifold}
    In \cite[Sec.~5]{bloch1992control}, sufficient conditions are given under which a smooth \concept{equilibrium submanifold} $N\subset 0_{\cD}\subset \cD$ of the zero section of $\cD$ can be stabilized for a certain broad subclass of nonholonomically constrained Lagrangian systems encompassed by \eqref{eq:nonhol-local} (or \eqref{eq:nonhol-global}).
    For this subclass, these sufficient conditions essentially amount to assuming that (i) the configuration space factors as a product $Q = Q_1 \times Q_2$ of smooth manifolds (without boundary), so that the configuration decomposes as $q = (q_1, q_2)$, and (ii) the determinants of a certain pair of matrix-valued functions are nowhere vanishing \cite[Eq.~15]{bloch1992control}.
    Restated in geometric language, the first determinant being nowhere vanishing is equivalent (by the implicit function theorem) to the existence of a smooth map $h\colon Q_2\to Q_1$ such that the equilibrium manifold
    \begin{equation}\label{eq:brm-assumption-1}
    N = \{(q_1, q_2)\colon q_1 = h(q_2)\}\eqqcolon \textnormal{graph}(h)
    \end{equation}
    is the graph of $h$ (identifying $N\subset 0_{\cD} \approx Q$ with a submanifold of $Q$ in the standard way), and the second determinant being nowhere vanishing is equivalent to the statement that
    \begin{equation}\label{eq:brm-assumption-2}
    \T Q|_N = \T N \oplus \cD|_{N}
    \end{equation}
    splits as the indicated direct (Whitney) sum.
    Under these and mild additional sufficient conditions, \cite[Thm~3]{bloch1992control} guarantees that $N$ can be rendered asymptotically stable by substituting a rather explicit smooth control force field $\T Q \to \T^* Q$ for $F$ in \eqref{eq:nonhol-local} (or \eqref{eq:nonhol-global}).
    As we now explain, this is entirely consistent with the non-stabilizability claim in item \ref{item:nonhol-lagrangian-2} of  Prop.~\ref{prop:nonhol-lagrangian} (with $\att = N$).
    First note \eqref{eq:brm-assumption-1} implies that $N$ is diffeomorphic to $Q_2$, so item \ref{item:nonhol-lagrangian-2} of Prop.~\ref{prop:nonhol-lagrangian} gives no information concerning stabilizability of $\att = N$ if either $Q_2$ is noncompact or if $Q_2$ is compact with $\chi(Q_2)=0$.
    On the other hand, it is easy to construct examples satisfying the sufficient conditions of \cite[Thm~3]{bloch1992control} regardless of the topology of $Q_2$, so one might worry that this contradicts the non-stabilizability claim in item \ref{item:nonhol-lagrangian-2} of Prop.~\ref{prop:nonhol-lagrangian} if $Q_2$ is compact and $\chi(Q_2)\neq 0$.
    However, there is no contradiction because Prop.~\ref{prop:nonhol-lagrangian} assumes the existence of a continuous vector field $Y$ which is nowhere $\cD$-valued, and such a $Y$ cannot exist if $Q_2$ is compact with $\chi(Q_2)\neq 0$.
    Indeed, if such a $Y$ existed then the image of $Y|_N$ under the linear projection $\T Q|_N = \T N \oplus \cD|_N \to \T N$ would yield a nowhere-vanishing vector field on $N \approx Q_2$, contradicting the Poincar\'{e}-Hopf theorem (Lem.~\ref{lem:poincare-hopf}) in light of the assumption that $\chi(N)=\chi(Q_2)\neq 0$. 
    Thus, Prop.~\ref{prop:nonhol-lagrangian} is in harmony with \cite[Thm~3]{bloch1992control}.
    However, Prop.~\ref{prop:nonhol-lagrangian} can still be used to obtain striking conclusions concerning concrete examples of nonholonomic systems such as the ``vertical rolling disk'' treated in Ex.~\ref{ex:vert-rolling-disk} below.
    \end{Rem}
    
    \begin{proof}[Proof of Prop.~\ref{prop:nonhol-lagrangian}]
    The claims in item \ref{item:nonhol-lagrangian-2} of the proposition follows directly from item \ref{item:nonhol-lagrangian-1}, Prop.~\ref{prop:general-second-order-cs-on-bundle-theorem}, and the Picard-Lindel\"{o}f theorem, so we need only prove the claim in item \ref{item:nonhol-lagrangian-1}.

    Define $\omega_L^\flat Y \coloneqq  i_Y\omega_L$ for a vector $Y$ and define the vector bundle isomorphism $\omega_L^\sharp \colon \T^*(\T Q) \to \T(\T Q)$ to be the inverse of $\omega_L^\flat \colon \T(\T Q)\to \T^*(\T Q)$; that $\omega_L^\flat$ is an isomorphism follows since the regularity of $L$ implies that $\omega_L$ is symplectic (nondegenerate) \cite[Prop.~3.5.9]{abraham1987foundations}. 
    Using the fact that $L$ is regular, a local coordinate computation using \cite[Prop.~3.5.6]{abraham1987foundations} reveals that  
    \begin{equation}\label{eq:ttqd-splitting-1}
    \forall v\in \cD\colon \T_v(TQ) = \T_v \cD \oplus \underbrace{\omega_L^\sharp (\D_v \pi_{\T Q})^*(\cD^0_{\pi_{\T Q}(v)})}_{E_v},
    \end{equation}
    so
    \begin{equation}\label{eq:ttqd-splitting-2}
    \T(\T Q)|_{\cD} = \T \cD \oplus E|_{\cD},
    \end{equation}
   where $E\subset \T(\T Q)$ is defined by \eqref{eq:ttqd-splitting-1}.
   Given any $Y\in \T(\T Q)|_{\cD}$, we uniquely decompose $Y$ as $Y = Y_{\T \cD} + Y_E$ according to the splitting \eqref{eq:ttqd-splitting-2}.
   Defining $G\coloneqq \omega_L^\sharp dE_L$ and $\widetilde{F}$ by $\widetilde{F}_v\coloneqq \omega_L^\sharp (\D_v \pi_{\T Q})^*F_v$ for $v\in \cD$, we may thus rewrite \eqref{eq:nonhol-global} as
   \begin{equation}\label{eq:nonhol-global-sharp-decomp}
   X = X_{\T \cD} \in G_{\T \cD} + G_{E} - \widetilde{F}_{\T \cD} - \widetilde{F}_E + E. 
   \end{equation}
   Because \eqref{eq:ttqd-splitting-2} is a direct sum, \eqref{eq:nonhol-global-sharp-decomp} implies that 
   \begin{equation}\label{eq:x-decomp-eliminate}
   X =  G_{\T \cD} - \widetilde{F}_{\T \cD}.
   \end{equation}
   If we define the linear projection $\Pi\colon \T\cD \oplus E|_{\cD} \to \T \cD$ onto the first factor, we can write \eqref{eq:x-decomp-eliminate} more explicitly as
  \begin{equation}\label{eq:x-decomp-eliminate-explicit}
  \ddot{q} = X_{\dot{q}} = \Pi((\omega_L^\sharp dE_L)_{\dot{q}} - \omega_L^\sharp (\D_{\dot{q}} \pi_{\T Q})^*F ), \qquad \dot{q} \in \cD.
  \end{equation} 
  This expression is smooth in $\dot{q}$ and $F$ since $\Pi$, $\omega_L$, $E_L$, and $\pi_{\T Q}$ are smooth.
  Since the map $F\colon \ctrl \to \T^* Q$ is assumed to be locally Lipschitz, so is the map $f\colon \ctrl \to \T(\T Q)$ defined by the right side of \eqref{eq:x-decomp-eliminate-explicit} (to define this map, substitute $\dot{q} = p(u)$ so that $\pi_{\T(\T Q)}\circ f = \iota \circ p$, where $\iota \colon \cD \hookrightarrow \T Q$ is the inclusion map).
  That the control system $(\ctrl, \T Q, p, f)$ is second order follows since the right side of \eqref{eq:x-decomp-eliminate-explicit} is the sum of the second order vector $\Pi((\omega_L^\sharp dE_L)_{\dot{q}})$ with the vertical vector $-\Pi(\omega_L^\sharp (\D_{\dot{q}} \pi_{\T Q})^*F_{\dot{q}} )$, so $\D \pi_{\T(\T Q)}\circ f = p$.
  This completes the proof. 
  \end{proof}
  
  \begin{Ex}[Vertical rolling disk]\label{ex:vert-rolling-disk}
  To illustrate Prop.~\ref{prop:nonhol-lagrangian} we consider the controlled \concept{vertical rolling disk} using steering and driving torque inputs, following \cite[Sec.~1.4]{bloch2015nonholonomic}.
  The configuration space $Q$ for this system is $Q = \SE(2) \times \sph^1$, with (``generalized'') coordinates $q = (x,y,\varphi,\theta)$.
  The Lagrangian for this system is equal to its total kinetic energy, namely:\footnote{For notational simplicity, in this example we view $\dot{q}\in \R^4$ and write $(q,\dot{q})$ in lieu of the more global notation explained in Footnote~\ref{foot:not-local-coord} and used earlier in this section.}
  \begin{equation}\label{eq:ver-rolling-disk-lagrangian}
  L(q,\dot{q}) = L(x,y,\varphi,\theta,\dot{x},\dot{y},\dot{\varphi},\dot{\theta}) = \frac{1}{2}m(\dot{x}^2 + \dot{y}^2) + \frac{1}{2}I\dot{\theta}^2 + \frac{1}{2} J \dot{\varphi}^2,
  \end{equation}
  where $m>0$ is the mass of the disk, $I>0$ is the moment of inertia of the disk about the axis perpendicular to the plane of the disk, and $J$ is the moment of inertia about an axis in the plane of the disk (both axes passing through the disk’s center) \cite[Eq.~1.4.1]{bloch2015nonholonomic}.
  If $R>0$ is the radius of the disk, the nonholonomic constraints of rolling without slipping are:
  \begin{equation}\label{eq:ver-rolling-disk-constraints}
  \begin{split}
  a^1\cdot (\dot{x},\dot{y},\dot{\varphi},\dot{\theta}) \coloneqq \dot{x} - R(\cos \varphi)\dot{\theta} &=0\\
   a^2\cdot (\dot{x},\dot{y},\dot{\varphi},\dot{\theta}) \coloneqq \dot{y} - R(\sin \varphi)\dot{\theta}&=0,
  \end{split}
  \end{equation}
  which state that a point $P_0$ fixed on the rim of the disk has zero velocity at its point of contact with the horizontal plane \cite[Eq.~1.4.2]{bloch2015nonholonomic}.
  Assuming we have controls in the directions of the two angles $\varphi$ and $\theta$, the Lagrange-d'Alembert equations \eqref{eq:nonhol-local} in the present case are:
  \begin{equation*}
  \frac{d}{dt}\frac{\partial L}{\partial \dot{q}} = u_\varphi g^\varphi + u_\theta g^\theta + \lambda_1 a^1 + \lambda_2 a^2,
  \end{equation*}
  where $g^\varphi = (0,0,1,0)$ and $g^\theta = (0,0,0,1)$ \cite[Eq.~1.4.3]{bloch2015nonholonomic}.
  Here $u_\varphi$ and $u_\theta$ are control inputs, so the control force of \eqref{eq:nonhol-local} and \eqref{eq:nonhol-global} is $F = u_\varphi g^\varphi + u_\theta g^\theta$, and the $\lambda_i$ are Lagrange multipliers chosen to ensure satisfaction of the constraints \eqref{eq:ver-rolling-disk-constraints}.

  The set $\cD$ of points $(x,y,\varphi,\theta,\dot{x},\dot{y},\dot{\varphi},\dot{\theta})$ satisfying \eqref{eq:ver-rolling-disk-constraints} is a smooth vector subbundle $\cD\subset \T (\SE(2) \times \sph^1)$, and the smooth vector field
  \begin{equation*}
  Y\coloneqq \frac{\partial}{\partial x}
  \end{equation*}
  is nowhere $\cD$-valued.
  Moreover, the Lagrangian \eqref{eq:ver-rolling-disk-lagrangian} is regular since the matrix
  \begin{equation*}
  \frac{\partial^2 L}{\partial \dot{q}\partial \dot{q}} = \begin{bmatrix}
  m & 0 & 0 & 0 \\
  0 & m & 0 & 0 \\
  0 & 0 & J & 0\\
  0 & 0 & 0 & I
  \end{bmatrix}
  \end{equation*}
  is invertible.
  Thus, if $\att\subset \cD$ is any compact subset having a well-defined (according to Def.~\ref{def:euler-cech}) and nonzero Euler characteristic $\chi(\att)$, then Prop.~\ref{prop:nonhol-lagrangian} implies that $\att$ cannot be made asymptotically stable for the closed-loop system determined by any locally Lipschitz control law $(q,\dot{q})\mapsto u_{\varphi}(q,\dot{q}),u_{\theta}(q,\dot{q})$.
  In particular, if $\att\subset \cD$ is a $2$-dimensional compact submanifold (without boundary), $\att$ cannot be made asymptotically stable by such feedback if $\att$ is not homeomorphic to either a $2$-torus or a Klein bottle.\footnote{This is because the only compact connected $2$-dimensional manifolds (without boundary) with zero Euler characteristic are the $2$-torus and the Klein bottle (see Lem.~\ref{lem:2d-manifolds-euler} in App.~\ref{app:low-dim-mfld-euler-zero}).}
  On the other hand, \cite[Prop.~2]{bloch1992control} gives a sufficient condition under which a $2$-dimensional equilibrium submanifold $N\subset 0_{\cD}\subset \cD$ (compact or not) can be made asymptotically stable for the vertical rolling disk; the preceding sentence implies that, if $N$ is compact, $N$ cannot satisfy these conditions if $N$ is not a torus or a Klein bottle (cf. Rem.~\ref{rem:eq-submanifold}). 

  Finally, if $S\subset \cD$ is any precompact subset having a well-defined and nonzero Euler characteristic, then the preceding considerations and Prop.~\ref{prop:nonhol-lagrangian} imply that $S$ cannot be rendered strictly positively invariant (Def.~\ref{def:strictly-inflowing}; cf. Def.~\ref{def:safety}) for the closed-loop system determined by any locally Lipschitz control law $(q,\dot{q})\mapsto u_{\varphi}(q,\dot{q}),u_{\theta}(q,\dot{q})$.
  \end{Ex}

\section{Comparison with selected point stabilization results} \label{sec:compare}  
Our motivation for Theorem~\ref{th:generalized-brockett} was to introduce a stabilizability test for compact subsets which are more general than single points; our motivation was not to sharpen existing stabilizability tests for single points. 
However, for completeness, in this section we compare  Theorem~\ref{th:generalized-brockett} with \cite[Thm~1.(iii)]{brockett1983asymptotic} and a weakened version of \cite[Thm~2]{coron1990necessary} in specialized settings in which (in particular) $\att = \{x_0\}$ is a single point.

The examples we present in this section are trivial from the perspective of control, since they essentially merely concern vector fields (control systems without control), but they are nonetheless adequate to compare the relative strengths of the three mentioned results.

\subsection{Comparison with Brockett's necessary condition}
In Rem.~\ref{rem:generalize-brockett} we explained that Theorem~\ref{th:generalized-brockett} is at least as strong as \cite[Thm~1.(iii)]{brockett1983asymptotic} in the special case that $\att = \{x_0\}$ is a point.
In this subsection we present an example in which a lack of stabilizability is detected by Theorem~\ref{th:generalized-brockett} but not by \cite[Thm~1.(iii)]{brockett1983asymptotic}.
Thus, Theorem~\ref{th:generalized-brockett} is \emph{strictly} stronger than \cite[Thm~1.(iii)]{brockett1983asymptotic} in the special case that $\att$ is a point.

\begin{Ex}\label{ex:strictly-stronger}
Consider the system of ordinary differential equations (Fig.~\ref{fig:ex-3-streams})
\begin{equation}\label{eq:ex-strictly-stronger-1}
\begin{split}
\dot{x} &= x^2-y^2\\
\dot{y} &=4xy^2
\end{split}
\end{equation}  
on $\R^2$.\footnote{Alternatively, in terms of complex numbers $z=x+iy\in \C$, with $i=\sqrt{-1}$, the right side of \eqref{eq:ex-strictly-stronger-1} can be written as $f(z) = \left(\frac{1}{2}+y\right)z^2 + \left(\frac{1}{2}-y\right)\bar{z}^2$, where $\bar{z}=x-iy$.}
We can view the vector field $(x,y)\mapsto f(x,y)$ defined by  \eqref{eq:ex-strictly-stronger-1} as the trivial control system $(\ctrl, \st, p,f) = (\R^2,\R^2, \id_{\R^2}, f)$ ``without control'', where the notation is as in \eqref{eq:control-system}.
Here, however, we identify vector fields including $f\colon \R^2 \to \T \R^2$ with maps $\R^2\to \R^2$ using the canonical identification $\T \R^2 \approx \R^2 \times \R^2$.
The origin of $\R^2$ is not asymptotically stable (hence not stabilizable) for $f$ since, e.g., it has index $0 \neq  1$.\footnote{For the standard definition of the index of an isolated equilibrium point see, e.g., \cite[p.~32]{milnor1965topology} or  \cite[p.~133]{guillemin1974differential}.}
We will show that (i) this is detected by Theorem~\ref{th:generalized-brockett}, but (ii) it is not possible to detect this using only \cite[Thm~1.(iii)]{brockett1983asymptotic}.

We begin with the latter claim.
Consider the equations
\begin{equation}\label{eq:ex-strictly-stronger-2}
\begin{split}
\delta &= x^2-y^2\\
\eta &=4xy^2
\end{split}
\end{equation}  
for any constant vector $(\delta, \eta)\in \R^2$.
From the first equation of \eqref{eq:ex-strictly-stronger-2}, $y^2 = x^2-\delta$; substituting this into the second equation yields $\eta = 4x(x^2-\delta)  = 4x^3-4\delta x$.
When $\eta = 0$, the solutions to the latter equation are $0,\pm \sqrt{|\delta|}$; because the function $x\mapsto 4x^3-4\delta x$ is increasing when $|x|>\sqrt{|\delta|}$, it follows that the latter equation always has a solution $x_*(\delta,\eta)$ satisfying $|x_*|> \sqrt{|\delta|}$ and $\lim_{(\delta,\eta)\to 0}x_*(\delta,\eta)=0$.
Since $x_*^2-\delta \geq 0$, it follows that $(x_*,\sqrt{x_*^2-\delta})$ is a solution to \eqref{eq:ex-strictly-stronger-2}.
Since $\delta, \eta$ can be made arbitrarily small, we see that Brockett's necessary condition is satisfied; thus, it is not possible to deduce from \cite[Thm~1.(iii)]{brockett1983asymptotic} that the origin is not stabilizable.

We now show that it \emph{is} possible to deduce that the origin is not stabilizable from Theorem~\ref{th:generalized-brockett}.
Consider, for $\epsilon > 0$, the equations 
\begin{equation}\label{eq:ex-strictly-stronger-3}
\begin{split}
\epsilon &= x^2-y^2\\
-\epsilon x&=4xy^2,
\end{split}
\end{equation}  
which correspond to the adversary $X_\epsilon= \epsilon(\frac{\partial}{\partial x} - x\frac{\partial}{\partial y})$ in the context of Theorem~\ref{th:generalized-brockett}.
\begin{figure}
	\centering
	\includegraphics[width=1.0\linewidth]{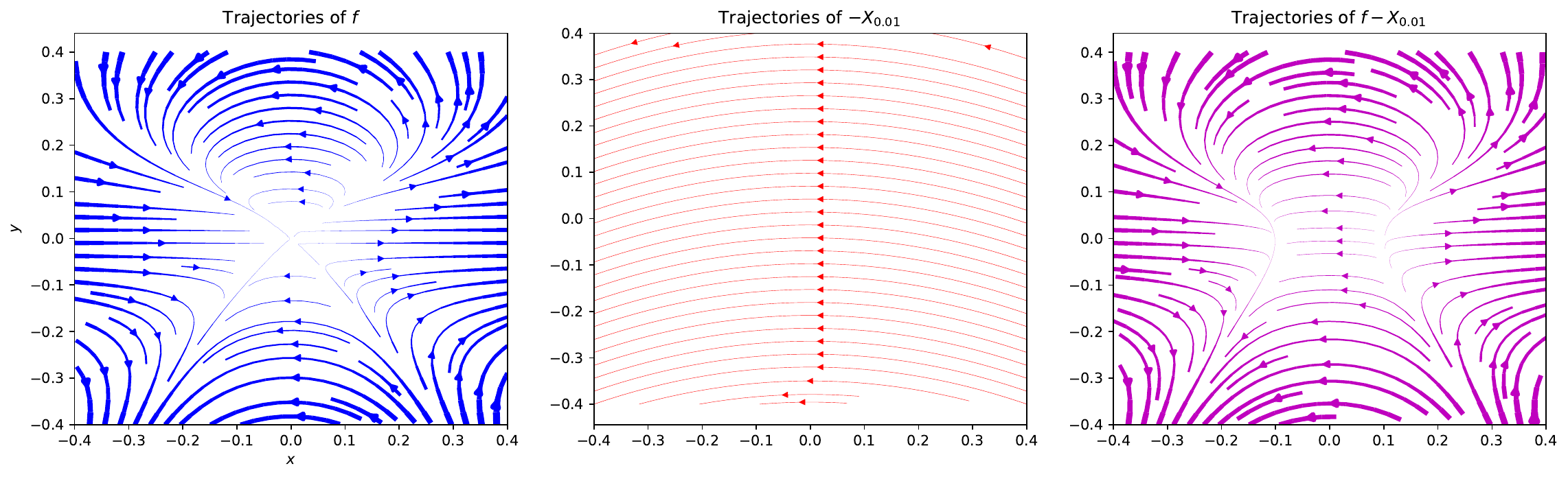}\caption{This figure illustrates Ex.~\ref{ex:strictly-stronger}.
	Shown are numerical approximations of trajectory segments of the vector fields: $f$ defined by \eqref{eq:ex-strictly-stronger-1} (left);  $-X_{\epsilon}$, with $\epsilon = 0.01$, defined by the left side of \eqref{eq:ex-strictly-stronger-3} (middle); and $f-X_{0.01}$ (right), illustrating the fact that $f-X_{\epsilon}$ has no equilibria for any $\epsilon > 0$. 
	For visualization purposes, the width of each trajectory is proportional to the norm of its instantaneous velocity, and the same proportionality constant was used in all three plots.
	Each plot was generated using the  function \texttt{streamplot} from the Python library matplotlib \cite{hunter2007matplotlib}.
	}\label{fig:ex-3-streams}
\end{figure}
From the second equation of \eqref{eq:ex-strictly-stronger-3}, 
\begin{equation}\label{eq:ex-strictly-stronger-4}
0 = x(4y^2+\epsilon).
\end{equation} 
Since the term in parentheses is strictly positive, \eqref{eq:ex-strictly-stronger-4} implies that $x = 0$.
Substituting this into the first equation of \eqref{eq:ex-strictly-stronger-3} yields $y^2 = -\epsilon$, which does not have a (real) solution.
Thus, \eqref{eq:ex-strictly-stronger-3} does not have a real solution $(x,y)$ for any $\epsilon > 0$.
Since $X_\epsilon\to 0$ as $\epsilon\to 0$ uniformly on any bounded neighborhood of the origin, Theorem~\ref{th:generalized-brockett} indeed implies that the origin is not stabilizable, as claimed. 
\end{Ex}        

\subsection{Comparison with Coron's necessary condition}\label{sec:coron}
In this subsection we first show (Prop.~\ref{prop:coron-compare}) that a weakened version of \cite[Thm~2]{coron1990necessary} is at least as strong as Theorem~\ref{th:generalized-brockett} in a special case that $\att$ is a point and the mild assumptions described in the paragraph below are satisfied.
We then present an example in which a lack of stabilizability of a point is detected by \cite[Thm~2]{coron1990necessary} but not by Theorem~\ref{th:generalized-brockett}.
Thus, the weakened version of \cite[Thm~2]{coron1990necessary} we present is \emph{strictly} stronger than Theorem~\ref{th:generalized-brockett} in the specialized setting of the present subsection.\footnote{Because Mansouri's theorems \cite[Thm~4]{mansouri2007local}, \cite[Thm~2.3]{mansouri2010topological}  for stabilizability of submanifolds of $\R^n$ generalize \cite[Thm~2]{coron1990necessary}, we also expect that the former theorems are strictly stronger than Theorem~\ref{th:generalized-brockett} in the corresponding specialized setting. 
We defer a careful comparison to future work.}

We begin by introducing the setting and some definitions from \cite{coron1990necessary} in order to state a somewhat weakened version of \cite[Thm~2]{coron1990necessary}.
Fix integers $n\geq 2$ and $m\geq 1$, let $\Omega\subset \R^n$ be an open neighborhood of $0$, and let $f\colon \Omega\times \R^m\to \R^n$ be a continuous map.
Denoting by $p\colon \Omega\times \R^m\to \Omega \subset \R^n$ the projection onto the first factor and identifying $f$ with a section of $\T \R^n \approx \R^n \times \R^n$, this defines a control system $(\Omega\times \R^m, \Omega, p, f)$ in the sense of \eqref{eq:control-system}. 
Given $\epsilon \in (0,\infty]$, define
\begin{equation}\label{eq:sigma-eps-def}
\Sigma_\epsilon \coloneqq \{(x,u)\in \Omega\times \R^m\colon f(x,u)\neq 0, \norm{x}< \epsilon\}.
\end{equation}
\begin{Rem}\label{rem:Sigma_eps}
Coron's definition of $\Sigma_\epsilon$ includes the additional stipulation that $\norm{u}< \epsilon$.
This is because Coron considers stabilizability of the origin for control systems $f$ satisfying $f(0,0)=0$ via control laws $x\mapsto u(x)$ satisfying $u(0) = 0$, a requirement which Coron can impose without loss of  generality in the setting of point stabilization.
However, it does not make sense to impose such a requirement in the more general context of our Theorem~\ref{th:generalized-brockett} (cf. \cite[p.~527]{mansouri2007local}).
For this reason, we have modified Coron's definition of $\Sigma_\epsilon$ (and also $\Omega$) in order to state a weakened version of \cite[Thm~2]{coron1990necessary} to facilitate comparison with Theorem~\ref{th:generalized-brockett}. 
\end{Rem}
Given an integer $k$, topological spaces $X$ and $Y$, and a continuous map $g\colon X\to Y$, $H_k(X)$ denotes the $k$-th singular homology of $X$ with coefficients in $\Z$ and $f_*\colon H_k(X)\to H_k(Y)$ denotes the induced homomorphism on homology.

\begin{\Thnon}[weakened version of {\cite[Thm~2]{coron1990necessary}}]
Let the control system $(\Omega\times \R^m, \Omega, p, f)$ and the set $\Sigma_\epsilon\subset \Omega \times \R^m\subset \R^n\times \R^m$ be as defined above.
Assume that the origin $\att = \{0\}$ is stabilizable (Def.~\ref{def:stabilizable}).
Then 
\begin{equation}\label{eq:weak-coron-condition}
\forall \epsilon \in (0,\infty]\colon f_*(H_{n-1}(\Sigma_\epsilon)) = H_{n-1}(\R^n\setminus \{0\}).
\end{equation}
\end{\Thnon}
We now show that this result is at least as strong as the corresponding specialization of Theorem~\ref{th:generalized-brockett}, in the sense that satisfaction of \eqref{eq:weak-coron-condition} implies satisfaction of the condition of Theorem~\ref{th:generalized-brockett} containing \eqref{eq:generalized-brockett-adversary} for the control system $(\Omega\times \R^m, \Omega, p, f)$.
When combined with Ex.~\ref{ex:artstein}, it follows that this weakened version of \cite[Thm~2]{coron1990necessary} is \emph{strictly} stronger than Theorem~\ref{th:generalized-brockett} in the special case of stabilizing a point in $\R^n$.
\begin{Prop}\label{prop:coron-compare}
Assume that \eqref{eq:weak-coron-condition} is satisfied in the setting of the weakened version of \cite[Thm~2]{coron1990necessary} stated above.
Then for any neighborhood $\cW\subset \Omega$ of $0\in \Omega \subset \R^n$, there exists a neighborhood $\cV\subset \R^n$ of $0\in \R^n$ such that, for any continuous adversary (vector field) $X\colon \cW\to \cV$,
\begin{equation}\label{eq:adversary-coron-compare}
f(\cW\times \R^m)\cap X(\cW)\neq \varnothing.
\end{equation} 
\end{Prop}
\begin{Rem}
Our proof is a minor modification of \cite[Sec.~2.(B)]{coron1990necessary}.
\end{Rem}
\begin{proof}
We assume \eqref{eq:weak-coron-condition} and want to show that the condition containing \eqref{eq:adversary-coron-compare} holds.
Let $\cW\subset \Omega\subset \R^n$ be any neighborhood of $0$ and fix $\epsilon > 0$ small enough that $B_\epsilon\subset \cW$, where $B_\epsilon\subset \R^n$ is the open ball of radius $\epsilon$ centered at $0$.
Let $K$ be a compact subset of $\Sigma_\epsilon\subset B_\epsilon\times \R^m$ such that 
\begin{equation}\label{eq:coron-1}
f_*(H_{n-1}(K))=H_{n-1}(\R^n \setminus \{0\}) \,\,(\cong \Z);
\end{equation}
the condition \eqref{eq:weak-coron-condition} implies the existence of such a compact set.
Since $K$ is a compact subset of $\Sigma_\epsilon$, there exists $\delta > 0$ such that
\begin{equation}\label{eq:coron-2}
\forall (x,u)\in K\colon  f(x,u)\not \in B_\delta.
\end{equation}

Since $B_\epsilon \subset \cW$, to show that the condition containing \eqref{eq:adversary-coron-compare} holds for some $\cV$ it suffices to establish the claim that, for any continuous adversary $X\colon  B_\epsilon \to B_\delta$, $X(B_\epsilon)\cap f(B_\epsilon\times \R^m) \neq \varnothing$ (taking $\cV= B_\delta$).
Indeed, if this is not the case then there exists a continuous adversary $X\colon B_\epsilon\to B_\delta$ such that
\begin{equation}\label{eq:coron-3}
\forall (x,u)\in B_\epsilon\times \R^m\colon f(x,u)\neq X(x).
\end{equation}
Let $\varphi\colon \R^n\to [0,1]$ be any continuous function satisfying $\varphi^{-1}(1)\supset B_\delta$ and $\varphi^{-1}(0)\supset f(K)$ \cite[Lem.~2.22]{lee2013smooth}, and define the continuous map $\theta\colon B_\epsilon \times \R^n\to \R^n$ by 
\begin{equation}\label{eq:theta-def}
\theta(x,y)\coloneqq y-\varphi(y)X(x).
\end{equation}
Since $sX(x)\in B_\delta$ for all $x\in B_\epsilon$ and $0 \leq s\leq 1$, it follows that $\theta(x,y) = 0$ if and only if $y\in B_\delta$ and $y = \varphi(y)X(x)$, which in turn holds if and only if $y = X(x)$.
Thus,
\begin{equation}\label{eq:coron-4}
\begin{split}
\forall x\in B_\epsilon, y\in \R^n\colon \theta(x,y)=0\iff y=X(x).
\end{split}
\end{equation}
It follows from \eqref{eq:coron-3} and \eqref{eq:coron-4} that $\theta(x,f(x,u))\neq 0$ is nonzero for all $x\in B_\epsilon$ and $u\in \R^m$, so the formula
$$h(t,(x,u))\coloneqq \theta(tx, f(tx,tu))$$
defines a continuous homotopy $h\colon [0,1]\times K\to \R^n\setminus \{0\}$ with $h(0,\slot)$ a constant map.
Moreover, \eqref{eq:theta-def} and the fact that $\varphi(f(K))=\{0\}$ imply that $h(1,\slot) = f|_{K}$, so $f|_K$ is nullhomotopic.
Thus, $f_*(H_{n-1}(K)) = \{0\}\subset H_{n-1}(\R^n\setminus \{0\})$, contradicting \eqref{eq:coron-1} and completing the proof. 
\end{proof}

\begin{Ex}\label{ex:artstein}
In this example we show that Theorem~\ref{th:generalized-brockett} can fail to detect that a point $\att = \{x_0\}$ is not stabilizable while the weakened version of \cite[Thm~2]{coron1990necessary} \emph{does} detect non-stabilizability.
Thus, this example together with Prop.~\ref{prop:coron-compare} imply that the weakened version of \cite[Thm~2]{coron1990necessary} stated above is strictly stronger than Theorem~\ref{th:generalized-brockett} under the assumptions of the present subsection.

Consider the system of ordinary differential equations (Fig.~\ref{fig:ex-4-streams})
\begin{equation}\label{eq:artstein}
\begin{split}
\dot{x} &= x^2-y^2\\
\dot{y} &= 2xy.
\end{split}
\end{equation}
on $\R^2$.\footnote{Alternatively, in terms of complex numbers $z=x+iy\in \C$, with $i=\sqrt{-1}$, the right side of \eqref{eq:ex-strictly-stronger-1} can be written as $g(z) = z^2$.}
Let $(x,y)\mapsto g(x,y)$ be the vector field defined by the right side of \eqref{eq:artstein}.
In order to apply the weakened version of \cite[Thm~2]{coron1990necessary}, which formally requires controls $u\in \R^m$ with $m\geq 1$, we fix any $m \geq  1$ and view \eqref{eq:artstein} as defining a trivial control system $(\R^2 \times \R^m, \R^2, p,g\circ p)$ ``without control'', where $p(x,y,u)\coloneqq (x,y)$ and the notation is as in \eqref{eq:control-system}, with $f = g\circ p$.
Here, however, we identify vector fields $\R^2 \to \T \R^2$ with maps $\R^2\to \R^2$ using the canonical identification $\T \R^2 \approx \R^2 \times \R^2$.
The origin of $\R^2$ is not asymptotically stable (hence not stabilizable) for $g$ since, e.g., it has index $2 \neq  1$.
We will show that (i) this is not detected by Theorem~\ref{th:generalized-brockett}, but (ii) it is detected by \cite[Thm~2]{coron1990necessary}.

We begin with the former claim.
Let $\cW\subset \R^2$ be an arbitrary neighborhood of the origin.
Since the origin is the unique zero of $g$ and since the origin has index $2\neq 0$ for $g$, there exists a neighborhood $\cV\subset \R^2$ of the origin such that, for any continuous adversary $X\colon \cW\to \cV$ taking values in $\cV$, the perturbed vector field $g-X$ has a zero in $\cW$.\footnote{This is because the index $2$ of $(0,0)$ for $g$ is the winding number (Brouwer degree \cite[p.~27]{milnor1965topology}) of the map $\frac{g}{\norm{g}}\colon \partial B \to \sph^1$, where $B$ is a small ball centered at $(0,0)$, and continuity implies that the winding number of $\frac{g-X}{\norm{g-X}}\colon \partial B \to S^1$ is well-defined and matches that of $\frac{g}{\norm{g}}$ if the norm of $X|_S$ is sufficiently small. Thus, the winding number of $\frac{g-X}{\norm{g-X}}$ is nonzero if $X$ is small enough, and this in turn implies that $g-X$ has at least one zero in $\interior(B)$ \cite[p.~28, Lem.~1]{milnor1965topology}.}
Thus, $f(p^{-1}(\cW))\cap X(\cW) = g(\cW)\cap X(\cW) \neq \varnothing$, so Theorem~\ref{th:generalized-brockett} cannot detect that the origin is not stabilizable (asymptotically stable for $g$).

\newcommand{\Rtz}{\R^2_{0}}
However, using the notation $\Rtz\coloneqq \R^2\setminus \{(0,0)\}$, the fact that the origin has index $2$ for $g$ implies that the induced map 
\begin{equation}\label{eq:g-star}
(g|_{\Rtz})_*\colon \Hom_1(\Rtz) \to \Hom_1(\Rtz)
\end{equation}
on singular homology sends a generator of $\Hom_1(\Rtz)\cong \Z$ to twice itself. 
In other words, $(g|_{\Rtz})_*\colon \Z\to \Z$ is the doubling map $k\mapsto 2k$, which is not surjective.
Since $p(\Sigma_\infty)= \Rtz$ and $f = g\circ p$, it follows that  $f_*(\Hom_1(\Sigma_\infty))\subset (g|_{\Rtz})_*(\Hom_1(\Rtz))\neq \Hom_1(\Rtz)$.
Thus, the weakened version of \cite[Thm~2]{coron1990necessary} implies that the origin is not stabilizable (asymptotically stable for $g$), as claimed. 
\begin{figure}
	\centering
	\includegraphics[width=0.35\linewidth]{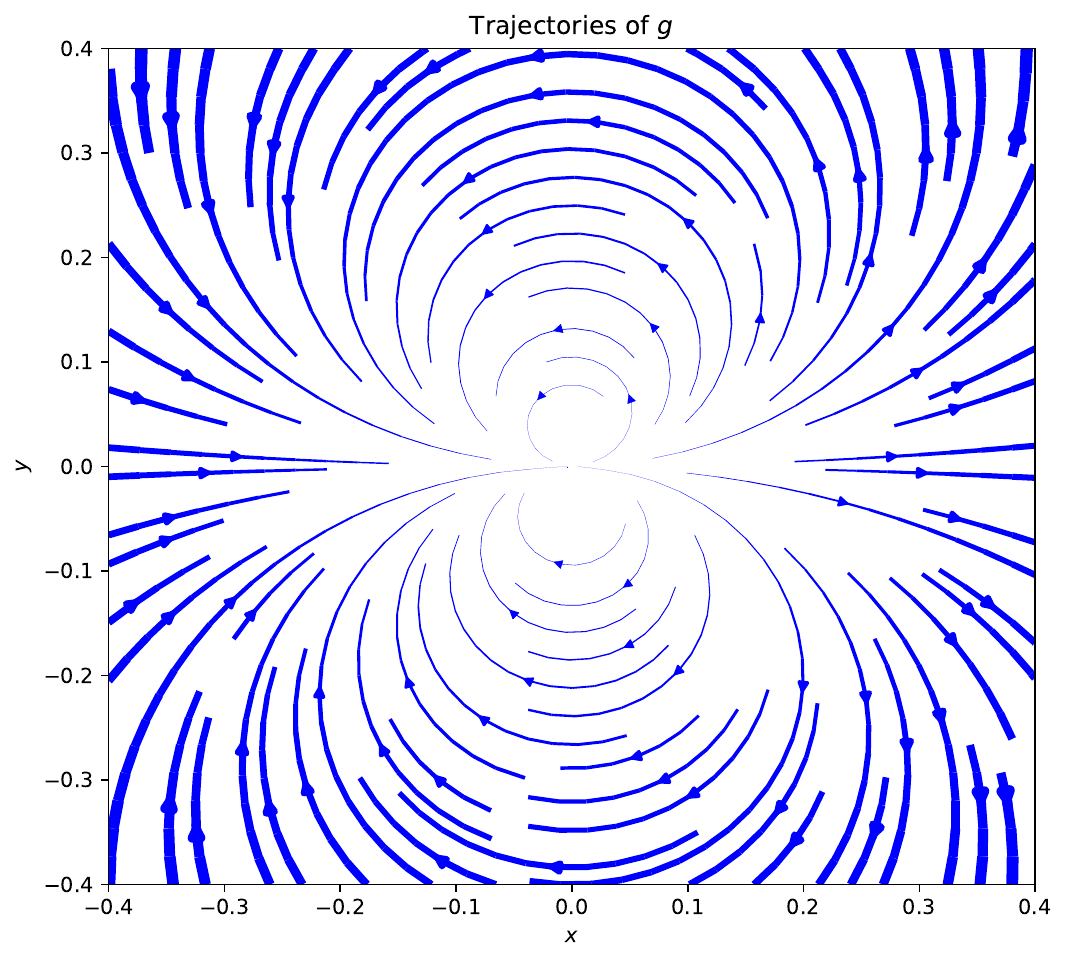}\caption{This figure illustrates Ex.~\ref{ex:artstein}.
	Shown are numerical approximations of trajectory segments of the vector field $g$ defined by the right side of \eqref{eq:artstein}. 
	For visualization purposes, the width of each trajectory is proportional to the norm of its instantaneous velocity.
	This plot was generated using the  function \texttt{streamplot} from the Python library matplotlib \cite{hunter2007matplotlib}.
	}\label{fig:ex-4-streams}
\end{figure}
\end{Ex}    

\section{Conclusion}\label{sec:conclusion}
We have generalized Brockett's necessary condition for feedback stabilization of points to one for feedback stabilization of general compact subsets having a nonzero Euler characteristic, where Euler characteristic is defined using \v{C}ech-Alexander-Spanier cohomology (Def.~\ref{def:euler-cech}).
This generalization furnishes a test which can be used to rule out stabilizability of a compact subset for the fairly general class of control systems \eqref{eq:control-system}.
Using this generalization, we have formulated an analogous necessary condition which can be used to test whether a control system can be made to operate safely relative to some subset of state space with a precompact complement having nonzero Euler characteristic.
As evidenced by \S \ref{sec:examples} and \ref{sec:applications}, both tests are readily applicable in a variety of concrete and fairly general situations.
However, especially for high-dimensional situations, it seems important to develop automated numerical approaches (perhaps partially based on \cite{kaczynski2004computational}) for performing the (co)homology and ``adversary'' computations needed to apply Theorems~\ref{th:generalized-brockett} and \ref{th:safety}.

In the special case that the compact subset under consideration for stabilization is a point, we showed in \S \ref{sec:compare} that our necessary condition (Theorem~\ref{th:generalized-brockett}) is strictly stronger than Brockett's (\cite[Thm~1.(iii)]{brockett1983asymptotic}), but is strictly weaker than one due to Coron (\cite[Thm~2]{coron1990necessary}) under certain mild assumptions.
Mansouri's necessary condition \cite[Thm~4]{mansouri2007local}, \cite[Thm~2.3]{mansouri2010topological} for stabilizability of submanifolds of $\R^n$ generalizes Coron's, and we expect that the former necessary condition is also strictly stronger than the corresponding specialization of ours (though ours does retain some advantages, as described in \S\ref{sec:related-work}).

Many systems (such as those in Ex.~\ref{ex:kinematic-unicycle-1}, \ref{ex:satellite}, \ref{ex:vert-rolling-disk}) evolve in non-Euclidean state spaces, to which Mansouri's results do not directly apply.
This raises the interesting prospect of generalizing Mansouri's results to non-Euclidean state spaces, and it would be similarly interesting to find generalizations of these results for stabilizing subsets more general than submanifolds (e.g., subsets such as a ``figure eight'') and for safety.
Arguably even more interesting is the prospect of devising necessary conditions which---unlike Theorems~\ref{th:generalized-brockett} and \ref{th:safety} and \cite[Thm~4]{mansouri2007local}, \cite[Thm~2.3]{mansouri2010topological}---can be used to test stabilizability and savability \emph{without} the assumption of nonzero Euler characteristic.
Such necessary conditions would be needed, for example, to test for stabilizability of topological circles (which have zero Euler characteristic) such as limit cycles.
Finally, it also seems important to develop analogous tests for discrete-time systems (building upon the work of \cite{lin1994design, kalabic2017mpc}) and, more generally \cite[Ex.~1]{kvalheim2021conley}, \emph{hybrid system}\footnote{See, e.g., \cite{simic2005towards,haghverdi2005bisimulation,Goebel_Sanfelice_Teel_2009,Johnson_Burden_Koditschek_2016,lerman2016category,culbertson2019formal}; more references can be found in \cite{kvalheim2021conley}.} models necessitated by the study of systems from robotics and biomechanics for which the making and breaking of contacts is an intrinsic feature \cite{koditschek2004mechanical, Koditschek_2021}.

   \subsection*{Dedication} We dedicate this paper to Anthony M. Bloch on the occasion of his 65th birthday.
   Kvalheim would like to thank Bloch for his mentorship and, in particular, for introducing him to Brockett's necessary condition and to geometric mechanics during an inspiring course taught by Bloch at the University of Michigan in 2014.
   Koditschek would like to thank Bloch for his inspirational work and many decades of kind, unstinting tutorial wisdom.

   \subsection*{Acknowledgments}
   This work is supported in part by the Army Research Office (ARO) under the SLICE Multidisciplinary University Research Initiatives (MURI) Program, award W911NF1810327, and in part by ONR grant N00014-16-1-2817, a Vannevar Bush Faculty Fellowship held by the second author, sponsored by the Basic Research Office of the Assistant Secretary of Defense for Research and Engineering.
   The authors gratefully acknowledge helpful conversations with Yuliy Baryshnikov,
   William Clark, George Council, Timothy Greco, Rohit Gupta, and Eugene Lerman. 
   We owe special gratitude to Clark for carefully reading the manuscript and making suggestions which improved its quality, and to Gupta for bringing relevant references to our attention. 
   Finally, we thank the two anonymous referees for useful suggestions.

   	\bibliographystyle{amsalpha}
   	\bibliography{ref}
   	
\appendix 

\section{Unique integrability, asymptotic stability, and Lyapunov functions}\label{app:uniq-int-lyap}
In this appendix we review some facts about continuous vector fields for the convenience of the reader. 
In addition to definitions, we recall here two key facts that we use: a continuous and uniquely integrable vector field (i) generates a unique maximal continuous local flow (Lem.~\ref{lem:cont-loc-int-vf-gen-local-flow}), for which (ii) any compact asymptotically stable subset possesses a (strict) proper $C^\infty$ Lyapunov function (Lem.~\ref{lem:converse-lyap}).  
   
\subsection{Unique integrability}\label{app:uniq-int}
Let $\st$ be a smooth manifold and $F$ be a continuous vector field on $\st$.
Consider the following ordinary differential equation:
\begin{equation}\label{eq:ode}
\dot{x} = F(x).
\end{equation}
For every initial condition $x_0\in \st$, the Peano existence theorem asserts existence (but not uniqueness) of a solution $t\mapsto x(t)$ of \eqref{eq:ode} defined on some open interval in $\R$ containing $0$ and satisfying $x(0)=x_0$ \cite[Sec.~II.2]{hartman2002ode}.
Such a solution is \concept{maximal} if it does not admit an extension to a solution of \eqref{eq:ode} defined on a strictly larger interval.
If maximal solutions of \eqref{eq:ode} happen to be unique, we say that the vector field $F$ is \concept{uniquely integrable}.
We remind the reader of the following fact, which follows from the Picard-Lindel\"{o}f theorem \cite[Sec.~II.1]{hartman2002ode}.
\begin{Rem}\label{rem:loc-lip-implies-uniq-int}
Every locally Lipschitz continuous vector field is uniquely integrable.
\end{Rem}

Solutions to an ordinary differential equation determined by a locally Lipschitz vector field depend continuously on the initial condition; this is also true when the vector field is merely continuous and uniquely integrable \cite[p.~94,~Thm~2.1]{hartman2002ode}.
Thus, in the same way that a smooth or locally Lipschitz vector field on a manifold gives rise to a unique maximal smooth or locally Lipschitz continuous local flow (e.g., see the proof of \cite[Thm~9.12]{lee2013smooth}), a continuous and uniquely integrable vector field $F$ gives rise to a unique maximal continuous local flow $\Phi\colon \dom(\Phi)\subset \R \times \st \to \st$.
Here $t\mapsto \Phi^t(x_0)$ is the (assumed unique) maximal solution to \eqref{eq:ode} with initial condition $x_0$; see \cite[pp.~211--212]{lee2013smooth} for the standard definition and properties of a maximal continuous local flow.	

We record this observation in the following.
\begin{Lem}\label{lem:cont-loc-int-vf-gen-local-flow}
If $F$ is a continuous and uniquely integrable vector field on $\st$, then there exists a unique maximal continuous local flow $\Phi$ on $\st$ generated by $F$.
\end{Lem}
	
\subsection{Asymptotic stability}\label{app:lyap}
Let $F$ be a continuous and uniquely integrable vector field on the smooth manifold $\st$.
A subset $\att \subset \st$ is \concept{invariant} if, for all $x_0\in \att$, the unique maximal solution to \eqref{eq:ode} with $x(0)=x_0$ is defined and belongs to $\att$ for all time ($x(t)\in \att$ for all $t\in \R$).
A compact invariant subset $\att\subset \st$ is \concept{Lyapunov stable} if, for every open set $U\supset \att$, there exists an open set $V\supset \att$ such that $x(t)\in U$ for all $t\geq 0$ when $x(0)\in V$.
A compact invariant subset $\att\subset \st$ is \concept{asymptotically stable} if $A$ is Lyapunov stable and there is an open set $W\supset A$ such that, if $x_0\in W$ and $x(t)$ is the unique maximal solution to \eqref{eq:ode} with $x(0)=x_0$, 
\begin{equation}\label{eq:att-converge}
\lim_{t\to\infty} \dist{x(t)}{A}=0.
\end{equation}
Here $\dist{\slot}{\slot}$ is any metric (distance) compatible with the topology of $\st$.
The \concept{basin of attraction} $B(\att)$ of an asymptotically stable subset $\att$ is an open set defined to be the largest possible set $W$ with the above property.
We say that $\att \subset \st$ is \concept{globally asymptotically stable} if $\att$ is asymptotically stable and $B(\att)=\st$.

A useful result is the following \cite[Lem.~1.6]{hurley1982attractors} (see also \cite{milnor1985conceptII,Milnor_2006}). 
Let $\Phi$ be the unique maximal continuous local flow generated by $F$ (Lem.~\ref{lem:cont-loc-int-vf-gen-local-flow}).
Then a compact invariant set $\att\subset \st$ is asymptotically stable if and only if $A$ has an open neighborhood $U$ with $[0,\infty)\times U\subset \dom(\Phi)$ and
\begin{equation}\label{eq:att-alt-def}
\att = \bigcap_{t>0}\Phi^{t}(U).
\end{equation}
This is one way (of several) to see that, despite the presence of the distance function in \eqref{eq:att-converge}, asymptotic stability of a \emph{compact} subset is a well-defined, metric-independent notion.

\begin{Rem}\label{rem:noncompact-attractor}
In this paper, we only discuss asymptotic stability for \emph{compact} subsets $\att\subset \st$.
If $A$ is noncompact then, as explained by Wilson \cite[pp.~425--426]{wilson1969smooth}, $A$ does not have a countable neighborhood basis, so a reasonable definition of asymptotic stability for $A$ requires the presence of a specific metric on $\st$.
If $A$ is compact, however, then asymptotic stability of $\att$ is a purely topological notion as demonstrated by the characterization \eqref{eq:att-alt-def} above.
\end{Rem}
If $\att$ is asymptotically stable for the continuous and uniquely integrable vector field $F$, we say that $V\colon B(\att)\to [0,\infty)$ is a (strict) proper $C^\infty$ \concept{Lyapunov function} for $\att$ if
\begin{itemize}
\item $V\colon B(A)\to [0,\infty)$ is $C^\infty$,
\item for each $c\in [0,\infty)$ the sublevel set $\{x\in B(\att)\colon V(x)\leq c\}$ is compact,
\item $\att = \{x\in B(\att)\colon V(x)=0\}$, and
\item the Lie derivative $L_F V$ satisfies $L_F V \leq 0$ and $L_F V(x)< 0$ when $x\not \in \att$.
\end{itemize}
We will use the following result.
\begin{Lem}[{\cite{wilson1969smooth,fathi2019smoothing}}]\label{lem:converse-lyap}
Let $\att\subset \st$ be a compact set which is asymptotically stable for the continuous and uniquely integrable vector field $F$ on $\st$.
Then there exists a proper $C^\infty$ Lyapunov function $V$ for $\att$.
\end{Lem}

\section{Low-dimensional manifolds with boundary and zero Euler characteristic}\label{app:low-dim-mfld-euler-zero}
In this appendix we explain why (in Lem.~\ref{lem:1d-manifolds-euler} and \ref{lem:2d-manifolds-euler}) the only compact connected manifolds with (or without) boundary having zero Euler characteristic and dimension less than or equal to two are the circle $\sph^1$, the torus $\tor^2$, the cylinder $\sph^1\times [0,1]$, the M\"{o}bius band (with boundary) $\mob$, and the Klein bottle $\klein$.
To do this, we rely on the standard classification theorems for compact connected manifolds of dimensions $1$ and $2$.
(The only connected $0$-dimensional manifold is a point, which has Euler characteristic equal to $1$.)

In the following two statements, a $C^0$ diffeomorphism means a homeomorphism.
\begin{Lem}\label{lem:1d-manifolds-euler}
Let $k\in \N_{\geq 0}\cup \{\infty\}$ and $\st$ be a compact connected $1$-dimensional $C^k$ manifold with (or without) boundary.
Then if the Euler characteristic $\chi(M)=0$, $\st$ is $C^k$ diffeomorphic to a circle. 
\end{Lem}
\begin{proof}
    First assume that $k = 0$ (resp. $k=\infty$).
    The classification theorem for topological (resp. $C^\infty$) $1$-manifolds with boundary implies that $\st$ is homeomorphic (resp. $C^\infty$ diffeomorphic) to $\sph^1$ or to $[0,1]$. 
    The former has zero Euler characteristic while the latter has Euler characteristic equal to $1\neq 0$.
    Thus, if $\chi(M)= 0$, $M$ must be homeomorphic (resp. $C^\infty$ diffeomorphic) to $\sph^1$.
    Finally, $M$ is $C^k$ diffeomorphic to a $C^\infty$ manifold in the case that $k\in \N_{\geq 1}$ \cite[p.~52,~Thm~2.10(a)]{hirsch1976differential}, so in that case the lemma follows from the case $k=\infty$.
    \end{proof}

    \begin{Lem}\label{lem:2d-manifolds-euler}
    Let $k\in \N_{\geq 0}\cup \{\infty\}$ and $M$ be a compact connected $2$-dimensional $C^k$ manifold with (or without) boundary $\partial M$.
    Then if the Euler characteristic $\chi(M)=0$, $M$ is $C^k$ diffeomorpic to one of the following: the $2$-torus $\tor^2$, the cylinder $\sph^1\times [0,1]$, the M\"{o}bius band (with boundary) $\mob$, or the Klein bottle $\klein$.
    \end{Lem}
    \begin{proof}
    We first note that every $C^k$ manifold with $k\in \N_{\geq 1}$ is $C^k$ diffeomorphic to a $C^\infty$ manifold \cite[p.~52,~Thm~2.10(a)]{hirsch1976differential} so that, as in the proof of Lem.~\ref{lem:1d-manifolds-euler}, it suffices to consider the case of $C^0$ and $C^\infty$ surfaces with boundary.

    Now, if such a surface $M$ with boundary is orientable, then the classification theorem for surfaces with boundary implies that $M$ is either homeomorphic or ($C^\infty$) diffeomorphic to a sphere or to a connected sum of $g\in \N_{\geq 1}$ tori with some finite number $b\in \N_{\geq 0}$ of open disks removed \cite[Thm~6.15,~Problem~6-5]{lee2010introduction}, \cite[p.~205,~Thm~3.7]{hirsch1976differential}. 
    Define $g$ to be zero in the case of the sphere.
    The Euler characteristic of a connected sum of $g$ tori is $2-2g$, and the Euler characteristic of a sphere is $2$ \cite[Prop.~6.19]{lee2010introduction}.
    Since the Euler characteristic of a disk is $1$ and the Euler characteristic of a circle is $0$, the inclusion-exclusion property of the Euler characteristic \cite[p.~205, Ex.~B.2]{spanier1966algebraic} implies the following formula:
    $$\chi(M) = 2-2g-b.$$
    We see that $\chi(M) = 0$ implies that either (i) $g = 1$ and $b = 0$ or (ii) $g = 0$ and $b = 2$.
    In the first case, $M$ is homeomorphic or diffeomorphic to $\tor^2$.
    In the second case, $M$ is homeomorphic or diffeomorphic to a sphere with two open disks removed, which is homeomorphic or diffeomorphic to $\sph^1\times [0,1]$.
    
    In the case that $M$ is nonorientable, $M$ is homeomorphic or diffeomorphic to the connected sum of $g\in \N_{\geq 1}$ projective planes with $b\in \N_{\geq 0}$ open disks removed \cite[Thm~6.15, Problem 6-5]{lee2010introduction}, \cite[p.~206, Thm~3.10]{hirsch1976differential}.
    The Euler characteristic of a connected sum of $g$ projective planes is $2-g$ \cite[Prop.~6.19]{lee2010introduction}, so---as in the orientable case---the inclusion-exclusion property of the Euler characteristic implies the formula 
    $$\chi(M) = 2-g-b.$$
    We see that $\chi(M) = 0$ implies that either (i) $g=1$ and $b = 1$ or (ii) $g=2$ and $b = 0$.
    Since the projective plane can be obtained by gluing a disk along its boundary circle to the boundary of a M\"{o}bius band \cite[Problem~6-2]{lee2010introduction}, \cite[p.~29, Ex.~15]{hirsch1976differential}, in the first case $M$ is homeomorphic or diffeomorphic to a M\"{o}bius band.
    In the second case $M$ is homeomorphic or diffeomorphic to a Klein bottle \cite[Lem.~6.16]{lee2013smooth}, \cite[p.~192]{hirsch1976differential}.
    This completes the proof.
    \end{proof}

\end{document}